\documentclass[reqno,10pt]{amsart}
\usepackage[dvipsnames,usenames]{color}
\usepackage[toc,page]{appendix}
\usepackage{cases,indentfirst}

    \usepackage{framed}
\usepackage{tikz}
\usepackage{mathcomp,wasysym}

\usepackage{hyperref}
\usepackage[mathscr]{euscript}

\usepackage{cite}
\usepackage{graphicx}
\usepackage{float}
\usepackage{subfigure}
\usepackage[all]{xy} \xyoption{arc} \xyoption{color}
\usepackage{epstopdf}
\usepackage{amsmath}
\usepackage{amsthm}
\usepackage{amsfonts}
\usepackage{amssymb}

\usepackage{verbatim}
\usepackage{overpic}

\numberwithin{equation}{section}

\newtheorem{theorem}{\sc Theorem}[section]
\newtheorem{lemma}{\sc Lemma}[section]
\newtheorem{proposition}{\sc Proposition}[section]
\newtheorem{remark}{\sc Remark}
\newtheorem{definition}{\sc Definition}[section]

\newcommand{\bset}[1]{ [#1] }

\newcommand{\pset}[1]{ (#1) }

\newcommand{\Pset}[1]{ \left(#1\right) }
\newcommand{\norm}[1]{ \|#1\| }

\newcommand{\abs}[1]{ |#1| }
\newcommand{\Abs}[1]{\left |#1\right| }

\newcommand{\cp}[1]{,_{#1}}

\def\hd{\bar{\partial}}

\def\local{l}

\def\UL{U_\local}

\def\thetal{\theta_\local}

\def\p{\partial}

\def\uu{u_0^\epsilon}

\numberwithin{figure}{section}

\title[Splash singularity for the magnetohydrodynamic equations]
{On the Splash Singularity for the free-boundary problem of the viscous and non-resistive incompressible magnetohydrodynamic equations in 3D}

\allowdisplaybreaks[4]
\begin{document}
\author[G.-Y. Hong]{Guangyi Hong}
\address{Guangyi Hong\newline\indent
School of Mathematics\newline\indent
South China University of Technology\newline\indent
Guangzhou 510641, China}
\email{magyhong@scut.edu.cn}

\author[T. Luo]{Tao Luo}\address{Tao Luo\newline\indent Department of Mathematics\newline\indent  City University of Hong Kong\newline\indent 
Kowloon, Hong Kong,  China}\email{taoluo@cityu.edu.hk}

\author[Z.-H. Zhao]{Zhonghao Zhao}
\address{Zhonghao Zhao\newline\indent
Department of Mathematics\newline\indent
City University of Hong Kong \newline\indent 
Kowloon, Hong Kong,   China}
\email{zhonghao2-c@my.cityu.edu.cn}
\begin{abstract}  
In this paper, the existence of finite-time splash singularity is proved for the free-boundary problem of the viscous and non-resistive incompressible magnetohydrodynamic (MHD) equations in $ \mathbb{R}^{3}$, based on a construction of a sequence of initial data alongside delicate estimates of the solutions. The result and analysis in this paper generalize those by Coutand and Shkoller in \cite[Ann. Inst. H. Poincar\'{e} C Anal. Non Lin\'{e}aire, 2019]{Shkoller-2019-APAN} from the viscous surface waves to the viscous conducting fluids with magnetic effects for which non-trivial magnetic fields may present on the free boundary. The arguments in this paper also hold for any space dimension $d\ge 2$.

\vspace*{6mm}


\vspace*{1mm}
\end{abstract}

\maketitle

\vspace{-2mm}


\section{Introduction}

In this paper, we consider the free boundary problem for the incompressible, viscous and non-resistive magnetohydrodynamic equations in $3$-dimensions:

\begin{subequations}
  \label{MHDe}
\begin{alignat}{2}
\tilde{\mathbf{u}}_t+ \tilde{\mathbf{u}}\cdot \nabla \tilde{\mathbf{u}} +  \nabla  P&= \frac{1}{4 \pi} \tilde{\mathbf{H}} \cdot \nabla \tilde{\mathbf{H}}+ \nu \Delta \tilde{\mathbf{u}}  \ \  \ &&\text{in} \ \ \Omega(t) \times [0,T],\label{Mhd-moment}\\
 \partial _{t}\tilde{\mathbf{H}}+ \tilde{\mathbf{u}} \cdot \nabla \tilde{\mathbf{H}}&=\tilde{\mathbf{H}} \cdot \nabla \tilde{\mathbf{u}} \ \  \ &&\text{in} \ \ \Omega(t) \times [0,T],\\
  {\operatorname{div}} \tilde{\mathbf{u}} &=0
&&\text{in} \ \ \Omega(t) \times [0,T], \\
{\operatorname{div}} \tilde{\mathbf{H}} &=0
&&\text{in} \ \ \  \Omega(t) \times [0,T],
\end{alignat}
\end{subequations}
describing the motion of the incompressible, viscous, conducting fluids with magnetic effects and without resistivity. Here $ \tilde{\mathbf{u}}=(\tilde{u} ^{1}, \tilde{u} ^{2},\tilde{u} ^{3}) $, $ \tilde{\mathbf{H}}=(\tilde{H}^{1},\tilde{H}^{2},\tilde{H}^{3}) $,  
and 
$\Omega (t) \subset \mathbb{R}^{3}$ represent the fluid velocity, the magnetic field and the moving domain, respectively. $P= p+ \frac{1}{8 \pi}\vert \tilde{\mathbf{H}}\vert ^{2}$ is the total pressure, where $ p $ represents the fluid pressure. We denote by $\mathbf{n}=\mathbf{n}(\mathbf{x}, t)$ the exterior unit normal vector to the free surface $\Gamma(t):=\partial \Omega(t)$, and by $ \mathcal{V} (\Gamma(t))$ the normal velocity of $\Gamma(t)$. Given a bounded smooth domain $ \Omega _{0} \subset \mathbb{R}^{3} $, we shall consider the equations \eqref{MHDe} supplemented with the following initial and boundary conditions: 
\begin{subequations}\label{free-bd}
\begin{numcases}
\displaystyle  \nu \operatorname{Def} \tilde{\mathbf{u}} \cdot \mathbf{n}+ \frac{1}{4 \pi}(\tilde{\mathbf{H}} \otimes \tilde{\mathbf{H}})\cdot \mathbf{n} - P\, \mathbf{n} = \mathbf{0},\ \ \tilde{\mathbf{H}} \cdot \mathbf{n}=0 \ \ &\text{on} \ \ $ \Gamma(t) $, \label{bd-mom}\\
\mathcal{V} (\Gamma(t)) = \tilde{\mathbf{u}} \cdot \mathbf{n} &\text{on} \ \ $ \Gamma(t) $, \\
\tilde{\mathbf{u}}   = \mathbf{u}_0,\ \tilde{\mathbf{H}}=\mathbf{H}_{0}  \ \  &\text{on} \ \ $ \Omega _{0} $.
  \end{numcases}
 \end{subequations}
 Since $\tilde{\mathbf{H}} \cdot \mathbf{n}=0$ implies that $(\tilde{\mathbf{H}} \otimes \tilde{\mathbf{H}})\cdot \mathbf{n}=0$ on $ \Gamma(t) $,  \eqref{bd-mom} reduces to 
 \[  \nu \operatorname{Def} \tilde{\mathbf{u}} \cdot \mathbf{n} - P\, \mathbf{n} = \mathbf{0},\ \ \tilde{\mathbf{H}} \cdot \mathbf{n}=0 \ \ \text{on} \ \  \Gamma(t) . \]
In view of \eqref{Mhd-moment} and \eqref{bd-mom}, we know that the pressure function $ P $ is determined by the following Dirichlet problem:
\begin{subequations}
  \label{p}
\begin{alignat}{2}
- \Delta P  &=  \tilde{u}^i,_j \tilde{u}^j,_i  - \frac{1}{4 \pi}\tilde{H} ^{i}, _{j}\tilde{H} ^{j}, _{i} \ \  \ &&\text{in} \ \ \Omega(t),\\
 P &=  \mathbf{n} \cdot \left[ \nu \operatorname{Def} \tilde{\mathbf{u}} \cdot \mathbf{n} \right] \ \ &&\text{on} \ \ \Gamma(t),
\end{alignat}
\end{subequations}
so that given an initial domain $\Omega _{0} $, an initial velocity field $\mathbf{u}_0$, and an initial magnetic field $ \mathbf{H}_{0} $, the initial pressure is obtained as the solution of (\ref{p}) at $t=0$.

The goal of the present paper is to study the splash singularity for the aforementioned free-boundary problem \eqref{MHDe}--\eqref{free-bd}. The study on the splash singularity of fluid interfaces can date back to the work \cite{Fefferman-2013-Ann-math} by Castro, C\'{o}rdoba, Fefferman, Gancedo and G\'{o}mez-Serrano, where the splash singularity for the two-dimensional inviscid water wave problem was shown to occur when a fluid interface remains locally smooth but intersects in finite time. In their proofs, a conformal map was used to transform the equations and the fluid domain. In \cite{Shkoller-2014-splash}, Coutand and Shkoller showed the existence of a finite-time splash for the three-dimensional one-phase incompressible Euler equations with free boundary based on an approximation of the self-intersecting fluid domain by a sequence of smooth fluid domains. While for the vortex-sheet problem described by the two-phase incompressible Euler equations, Fefferman et al. \cite{Fefferman-2016-Duke-twophase} and Coutand and Shkoller \cite{Coutand-Shkoller-2016-ARMA} showed that no splash singularity can occur in finite time while the interface remains locally smooth. We also remark that except for the splash singularity, there is another singularity, so-called splat singularity, which can occur in inviscid flows, for some recent works in this direction, one may see, for instance, \cite{Fefferman-2013-Ann-math} and the references therein. 

When the viscosity is taken into account, the strategy for the study on splash singularity of the inviscid flows mentioned above could not work since it relies on the ability to flow backward-in-time. By using the transformation of the fluid domain employed in \cite{Fefferman-2013-Ann-math} alongside stability estimates, Castro et al. \cite{Castro-2019-AnnPDE} proved the existence of finite-time splash singularities for the free boundary-problem of 
incompressible Navier-Stokes equations in 2D. Later on, Coutand and Shkoller \cite{Shkoller-2019-APAN} developed a new method to show the existence of finite-time singularities for the free-boundary problem of the incompressible Navier-Stokes equations for any dimension $ d \geq 2 $. Motivated by the results in \cite{Castro-2019-AnnPDE,Shkoller-2019-APAN}, there have been some works on the splash singularity for the free-boundary viscoelastic fluids of the Oldroyd-B type (cf. \cite{marcati-oldroyd-B-ARMA,marcati-oldroyd-B-NoDEA-2017,marcati-Advance-2020}). However, as far as we are concerned, there are only few mathematical research on the singularity formation for the free-boundary problems of MHD equations. Indeed, it is quite important and interesting to investigate the singularity formation in the evolution of free surfaces for MHD problems. In MHD, the magnetic tension force,  $ \frac{1}{4 \pi} \tilde{\mathbf{H}} \cdot \nabla \tilde{\mathbf{H}}$ appearing on the right hand side of the first equation of \eqref{MHDe}, as a part of Lorentz force, plays a very different role other than the pressure. The force played by the magnetic tension is to straighten bent magnetic field lines due to which completely new wave phenomena without analogue in the ordinary fluid theory may present. It should be noted, the free surface $\Gamma(t)$ for the free-boundary problem \eqref{MHDe}--\eqref{free-bd} is foliated by magnetic lines since $\tilde{\mathbf{H}} \cdot \mathbf{n}=0$ on the free surface $\Gamma(t)$. Therefore, the force of magnetic tension is an essential part to drive the motion of the free surface
for the MHD problem. Recently, Hao-Yang \cite{hao2023splash} proved the finite-time splash singularity for the free-boundary problem of the two-dimensional incompressible MHD equations by using the conformal map employed in \cite{Fefferman-2013-Ann-math}. The approach of adopting the conformal mapping does not work for the three dimensional case. In the present paper, we shall focus on the existence of finite-time splash singularity for the incompressible MHD equations in three dimensions.  Moreover, we work in this paper for solutions for which the estimates of higher order energy are given. This is different from those used in \cite{hao2023splash} where the norms in Beale spaces based on interpolation introduced in \cite{Beale-1981} are used. The approach and estimates used the present paper hold for any space dimension $d\ge 2$. Dealing with the strong coupling of velocity, magnetic field, evolution of free surface and effects of viscosity is one of main themes of the paper. 

It is worth noting that, despite the lack of study on the splash singularity for MHD equations, the free-boundary problem of the MHD equations has been subjects of many mathematical studies because of its physical importance, complexity, rich phenomena, and mathematical challenges. Hao and Luo \cite{Hao-Luo-2014-ARMA} established the \emph{a priori} estimates for the free-boundary problem of the incompressible inviscid MHD equations in a general smooth initial domain and without surface tension under the Taylor sign condition. They \cite{Hao-Luo-2020-CMP} also proved the ill-posedness of the free-boundary problem in a 2D domain when the Taylor sign condition is violated. The local existence results on solutions to the free-boundary problem of the incompressible inviscid MHD equations were established in \cite{Gu-Wangyanjin-JMPA} and \cite{gu2021local} for the case with surface tension and the case without surface tension, respectively. A local existence result on a linearized free-boundary problem of incompressible MHD equations in a general initial domain was established in \cite{Hao-TAO-2021-JDE}. The local-in-time well-posedness of the free-boundary problems of compressible inviscid MHD with or without surface tensions was established by Trakhinin and Wang in \cite{TW1} and \cite{TW2}, respectively. The study of evolutionary free interfaces such as current-vortex sheets or plasma-vacuum interfaces for MHD equations for which magnetic fields play essential role to the well-posedness and stability is very interesting and active. Important progresses have been made on related  research topics. One may refer to \cite{CWY,CMST, ST, SWZ,SWZ1, T1,T2,T3, WX1,WX2} for the inviscid theory and \cite{DJJ, JJW} for viscous theory. For results on the free-boundary problems of fluids, one may refer to \cite{wusijue-1997-Invention,wusijue-1999-JAMS,Shkoller-2007-Jams,zhangping-zhangzhifei-2008-CPAM,Chri-Lindblad-2000,Lindblad-2005-Ann-of-Math,David-Lannes-2005,Zeng-chognchun-Shatah-2008-CPAM,Alazard-2014} for the inviscid theory of the incompressible Euler equations, and \cite{Solonnikov-1977,Solonnikov-1992,Beale-1981,Guo-Tice-2013-local,Guiguilong-Peking-2021} for the viscous theory of incompressible Navier-Stokes equations. 

To proceed, we quote the definition of splash singularity from \cite{Shkoller-2019-APAN}.
\begin{definition}A point of self-intersection is defined to be \emph{splash} singularity at time $T$ if a locally smooth, time-dependent fluid interface or free-boundary self-intersects at a point at time $T$ and its local smoothness is conservative.
\end{definition}
We shall prove that there exist smooth initial data for the free boundary problem of the viscous and non-resistive incompressible MHD equations \eqref{MHDe} for which such a splash singularity occurs in finite time. Precisely, our main results can be stated as follows.


\begin{theorem}[Finite-time splash singularity]\label{mainThm} There exists an initial data set $ (\Omega _{0}, \mathbf{u} _{0},\mathbf{H} _{0}) $ such that a finite-time splash singularity forms for any smooth solution $ (\Omega(t), \mathbf{u},\mathbf{H}) $ to the free-boundary problem \eqref{MHDe}--\eqref{free-bd} ; that is, after a finite time $T^*>0$, the interface $\Gamma(T^*)=\p \Omega(T^*) $ self-intersects. The constructed initial data set $ (\Omega _{0}, \mathbf{u} _{0},\mathbf{H} _{0}) $ satisfy 
\begin{enumerate}
\item  the initial domain $\Omega _{0}\subset \mathbb{R}  ^3$ is open bounded and in $C^ \infty $-class, 
\item the initial initial velocity  $\mathbf{u}_0$ is a smooth divergence-free vector field satisfying the compatibility condition
$$\left[ \operatorname{Def} \mathbf{u}_0 \cdot \mathbf{N}\right] \times \mathbf{N} =0 \text{ on } \p \Omega_0 ,$$
 where $\mathbf{N}$ is the unit outer normal to  $\p \Omega _{0}$, 
\item the initial magnetic field $ \mathbf{H}_{0}$ is a smooth divergence-free vector field satisfying $\mathbf{H}_{0}\cdot \mathbf{N} =0 \text{ on } \p \Omega_0$.
\end{enumerate}
\end{theorem}
The proof of Theorem \ref{mainThm} is carried out in spirit of \cite{Shkoller-2019-APAN}. The key of matter is to construct a sequence of initial data which is characterized by the parameter $ \epsilon >0$, as well as to establish some uniform-in-$ \epsilon $ estimates for the solution which allow us to extract a common lifespan of the solutions as $ \epsilon \rightarrow 0 $. As in \cite{Shkoller-2019-APAN}, we choose a dinosaur domain as the ``reference domain'', and construct a sequence of initial data based on it. A crucial ingredient of our analysis lies in the construction of the initial magnetic field, via the div-curl system alongside the carefully chosen \emph{curl} function, which, in particular, allows for non-zero boundary value. For details, please see Section \ref{sec:the_initial_magnetic_field}. To derive the uniform-in-$ \epsilon $ estimates for the solutions, as in \cite{Shkoller-2019-APAN}, we localize the equations near the boundary with the boundary local charts and establish the boundary regularity of the velocity, and then recover the regularity of the velocity in the fluid domain via the Stokes system. However, the estimates for the velocity must involve the ones for the magnetic field. For this, we derive some delicate estimates for the magnetic field carefully in the initial domains $ \Omega ^{\epsilon} $ by utilizing the transport character of the magnetic field and the coupling between the magnetic field equation and the momentum equation.


Finally, the rest of the paper is organized as follows: 
In Section \ref{sec:notation}, we define some notations and recall some known facts for later use. In Section \ref{sec-initial_data}, we construct the sequence of initial data, and introduce some regularity results and Sobolev inequalities relevant to the constructed initial domain. In Section \ref{sec-a_priori-esti}, we reformulate the problem with the Lagrangian coordinates as well as the local boundary charts, and establish the desired estimates for the solution. In Section \ref{section8}, we give the proof of the main theorem.


\vspace{6mm}

\section{Preliminaries} \label{sec:notation}
In this section, we introduce some notations and recall some known facts that will be frequently used later.

\subsection{Some notations} \label{sec:grad-horiz-deriv}
Throughout this paper, we use
$
\nabla =\left( \frac{\p}{\p x_1}\,,  \frac{\p}{\p x_2}\,,  \frac{\p}{\p x_3}  \right)
$ to denote the $3$-dimensional gradient vector, and $\operatorname{Def} \mathbf{u} = \nabla \mathbf{u} +  \nabla \mathbf{u}^T$ to denote the twice of the symmetric part of the gradient of velocity.
We use $f\cp{k}=\frac{\partial f}{\partial x_k}\ \mbox{or} \ \frac{\partial f}{\partial y_k}$ as the symbol of the $k$th partial derivative of $f$. Repeated indices $i,j,k$, etc., are summed from $1$ to $3$. For example, $F\cp{ii}=\sum_{i=1}^3\frac{\p^2F}{\p x_i\p x_i}$, and $F^i\cp{\alpha} I^{\alpha\beta} G^i\cp{\beta}=\sum_{i=1}^3\sum_{\alpha=1}^{2}\sum_{\beta=1}^2\frac{\p F^i}{\p x_\alpha} I^{\alpha\beta} \frac{\p G^i}{\p x_\beta}$. Furthermore, we denote by $ \mathbb{I}_{3} $ the $ 3 \times 3 $ identity matrix, and denote by $ \partial ^{k}f $ any $ k $th-order derivative of $ f $.

\subsection{Local charts for smooth domains}\label{localchart}
let $ B=B(0,1) $ denote the open unit ball in $\mathbb{R}^3$ centered at the origin, and $ B ^{+}=B \cap \left\{ y _{3}> 0\right\} $, $ B ^{0}=\bar{B}\cap \left\{ y _{3}=0 \right\} $. Given a smooth open set $ \Omega \subset \mathbb{R}^{3} $, if there exists a collection of open sets $ \mathcal{U}:=\left\{ U _{l} \right\} _{l=1}^{L} $ such that $ U _{l}\subset \Omega (l=K+1,\cdots,L) $ for some $ K<L $, and for each $ l \in \left\{ 1,2,\cdots,L \right\}  $, there exists a $ C ^{\infty} $ maps $ \theta _{l} $ which is a $ C ^{\infty} $ diffeomorphism from $B$ to $U_{L}$, and 
\begin{gather*}
\displaystyle \theta _{l}(B ^{+})=U _{l}\cap \Omega,\ \ \theta _{l}(B ^{0})=U _{l}\cap \Gamma\ \ \mbox{} l=1,2,\cdots,K,\\ 
\displaystyle  \Gamma=\p \Omega \subset\cup _{l=1}^{K}U _{l},\ \  \ \Omega \subset \cup _{l=1}^{L}U _{l}.
\end{gather*}
Then we call $ \left( \mathcal{U}, \cup _{l=1}^{L}\theta _{l} \right)  $ the local charts of $ \Omega $. Moreover, one may assume that $ \mathop{\mathrm{det}} \nabla \theta _{l}=C _{l} $ for some constant $ C _{l} >0$.

\def\R{ \mathbb{R}  }

\subsection{Tangential (or horizontal) derivatives}\label{sec: tangential derivative}We define the tangential derivatives for boundary charts $\UL\cap\Omega$, for $1\le\local\le K$. The $\alpha$th-component of the tangential derivatives on the $l$th boundary chart is defined to be:
\begin{align*}
	\bar \p_ \alpha  f=\Pset{\frac{\p}{\p y_\alpha}\bset{f\circ\thetal}}\circ\thetal^{-1}=\Pset{\pset{ \nabla  f\circ\thetal}\frac{\partial\thetal}{\partial y_\alpha}}\circ\thetal^{-1} \,,
\end{align*}and the operator $\bar{\partial} = (\bar \partial_{1},  \bar \partial_{2})$. For functions defined directly on  $B^+$, $\hd$ is simply the horizontal derivative $\hd = (\partial_{y_1},  \partial_{y_{2}})$.

\subsection{Sobolev spaces} \label{sec:diff-norms-open}
The Sobolev space $H^k(U)$ is defined to be the completion of $C^\infty(\bar{U})$ $\pset{C^\infty(\bar{U}; \mathbb{R}  ^3)}$ in the norm
\begin{align*}
	\norm{u}_{k,U}^2=\sum_{\abs{a}\le k}\int_U \Abs{ \nabla ^a u(\mathbf{x}) }^2 \mathrm{d}\mathbf{x},
\end{align*}
for integers $k\ge0$ and a bounded domain $U$ of $\R^3$, with the multi-index $a\in \mathbb{Z}  ^3_+$, such that $\abs{a}=a_1+ a_2+a_3$ ($a_j$ is the order of $j$th derivative).
For real numbers $s\ge0$, we define the Sobolev spaces $H^s(U)$ and the norms $\label{n:interior norm}\norm{\cdot}_{s,U}$ by interpolation.
For simplicity, $H^s(U)$ is used to represent $H^s(U; \mathbb{R}^3)$ if there is no possibility for confusion.

\subsection{Sobolev spaces on a surface} \label{sec:sobolev-spaces-gamma}The Sobolev norm on a surface $\Gamma$ is defined as
\begin{align*}
	\norm{u}_{k,\Gamma}^2=\sum_{\abs{a}\le k } \int_\Gamma \Abs{ \hd^a u(\mathbf{x})}^2 \mathrm{d}S,
\end{align*}
for functions $u\in H^k(\Gamma)$, $k\ge0$ and a multi-index $a\in \mathbb{Z}  ^{2}_+$.
For real $s\ge0$, we define the Hilbert space $H^s(\Gamma)$ and the boundary norm $\| \cdot \|_s$ by interpolation. And we define $H^{-s}(\Gamma)=H^s(\Gamma)' $ for real $s\ge0$.

\vspace{6mm}

\section{Construction of the initial data} 
\label{sec-initial_data}


\subsection{The sequence of  initial  domains}\label{sec::dino_wave}

We shall adapt the construction of initial domains from \cite{Shkoller-2019-APAN}.  As in \cite{Shkoller-2019-APAN}, we first define a dinosaur domain $ \Omega $. 

\begin{definition}[The domain $ \Omega $]\label{Defi-OMDA} The domain $ \Omega \subset \mathbb{R}^{3} $, is a smooth bounded domain (see Figure \ref{Fig-dinau}(a)) with boundary $ \Gamma $  containing three particular open subsets of $ \Omega $ as follows: 
\begin{enumerate}
  \item[1.] an open subset $ \omega \subset \Omega $ whose boundary $ \partial \omega $ is a vertical circular cylinder of radius $ 1 $ and length $ h $, 

  \item[2.] an open subset $ \omega _{+}\subset \Omega $ which is the lower-half of an open ball of radius $ 1 $, located directly below the cylindrical region $ \omega $, and in contact with the cylindrical region $ \bar{\omega} $. The ``south pole'' of $ \omega _{+} $ is the point $ \mathbf{X} _{+} $, 

  \item[3.] an open subset $ \omega _{-} $ directly below, at a distance $ 1 $, from the point $ \mathbf{X} _{+} $ of $ \omega _{+} $, such that the point with maximal vertical coordinate in $ \partial \omega _{-} \cap \Gamma$ form a subset of the horizontal plane $ x _{3}=0 $, 

  \item[4.] let ${\bf 0}$ be the origin of $ \mathbb{R}^{3} $, we assume that $ {\bf 0} \in \partial \omega _{-}\cap  \left\{ x _{3}=0 \right\}$. And $ \mathbf{X} _{+} =(0, 0, 1)$.
\end{enumerate}
\end{definition}
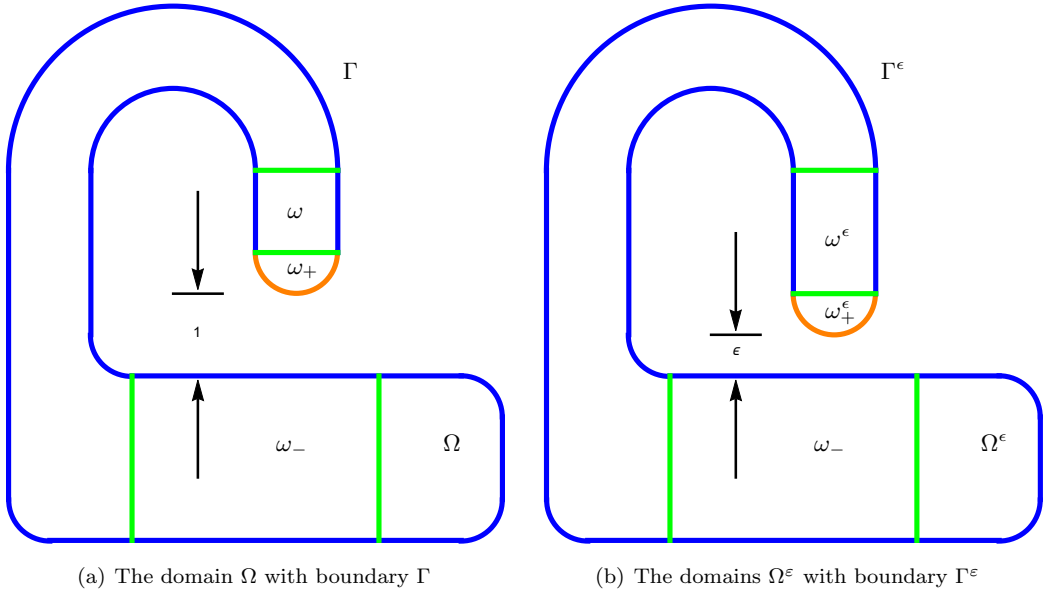
\begin{figure}[H]\label{Fig-dinau}
\centering
\subfigure[The domain $\Omega$ with boundary $\Gamma$]
{
\begin{overpic}[width=0.45\textwidth]{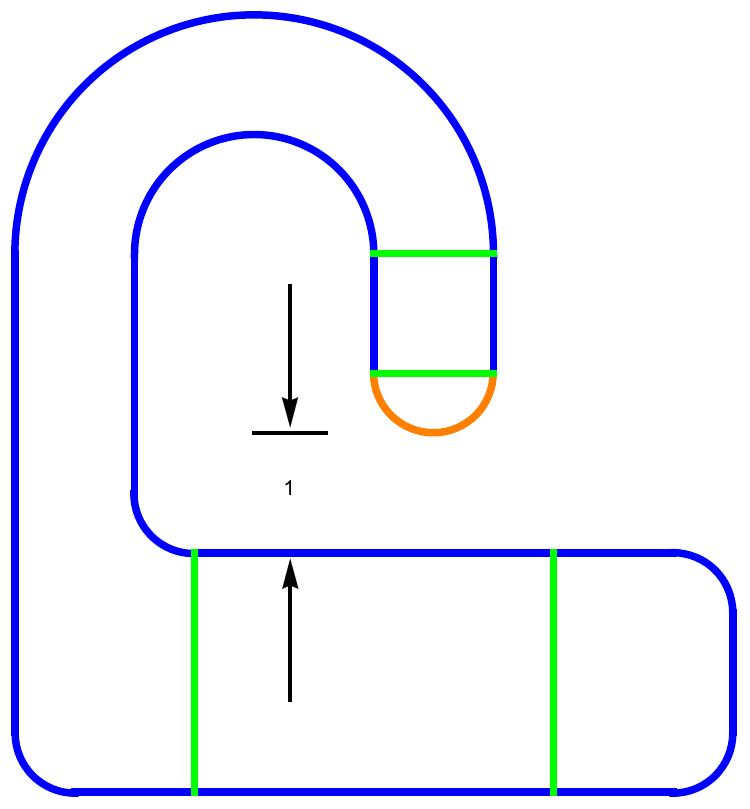}
  \put(62,85){\small\bfseries \color{black}{$\Gamma$}}
  \put(52,60){\small\bfseries \color{black}{$\omega$}}
  \put(52,50){\small\bfseries \color{black}{$\omega_+$}}
  \put(50,18){\small\bfseries \color{black}{$\omega_-$}}
  \put(80,18){\small\bfseries \color{black}{$\Omega$}}
\end{overpic}
}
\subfigure[The domains $\Omega^\varepsilon $ with boundary $\Gamma^ \varepsilon$]
{
\begin{overpic}[width=0.45\textwidth]{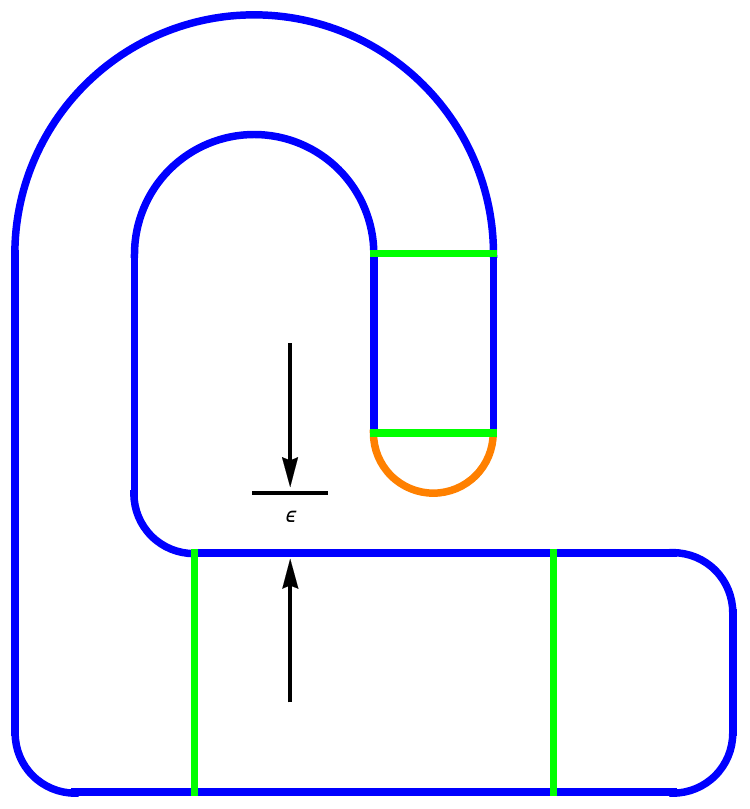}
  \put(62,85){\small\bfseries \color{black}{$\Gamma^\epsilon$}}
  \put(52,55){\small\bfseries \color{black}{$\omega^\epsilon$}}
  \put(52,42){\small\bfseries \color{black}{$\omega_+^\epsilon$}}
  \put(50,18){\small\bfseries \color{black}{$\omega_-$}}
  \put(80,18){\small\bfseries \color{black}{$\Omega^\epsilon$}}
\end{overpic}}
\caption{Fig The cross-section of the "dinosaur wave" domains}
\end{figure}

We now proceed to the definition of the sequence of initial domains $\Omega ^{\epsilon} $.
\begin{definition}[The initial domains $\Omega^ \epsilon $]\label{def-dino-e} Let $ \Omega $ be as in Definition \ref{Defi-OMDA}. For $ 0 < \epsilon \ll 1$, as shown in Figure \ref{Fig-dinau}(b), the $\epsilon$-modified domain  $\Omega ^{\epsilon} \subset \mathbb{R}  ^3$, with boundary
$\Gamma^ \epsilon $, is a smooth bounded domain with the following properties: 
\begin{enumerate}
\item[1.]  $\Omega ^{\epsilon}$  contains  an open subset $\omega^ \epsilon \subset \Omega^ \epsilon $, which is a vertical dilation of the domain $\omega$,
whose boundary
$\partial\omega^ \epsilon \cap \Gamma^ \epsilon $ is a vertical circular cylinder of radius $r$ and of length $h+1- \epsilon $,

\item[2.] $\Omega ^{\epsilon}$ contains an open subset $\omega_+^ \epsilon \subset\Omega^ \epsilon $, obtained by translating the set $\omega^+$ downward a distance $1- \epsilon $ vertically. Therefore, $\omega_+^ \epsilon $ is the lower-half of an open ball of radius $1$,
 below $\omega^ \epsilon $ directly. Denote the  ``south pole'' of $\omega_+^ \epsilon $ by $\mathbf{X}_+^ \epsilon $, 

\item[3.] $\Omega ^{\epsilon}$ contains an open subset $\omega_-\subset\Omega^ \epsilon $ which is a distance $ \epsilon $ from, and below directly, the point $\mathbf{X}_+^ \epsilon $,
such that the points with maximal vertical coordinate in $\partial\omega_-\cap \Gamma ^{\epsilon}$ form a subset of the horizontal plane $x_3=0$.
It is assumed that a $2$-dimensional ball of radius $\sqrt{ \epsilon }$ is contained in $\partial\omega_-\cap \Gamma ^{\epsilon}$,

\item[4.] let ${\bf 0}$ be the origin of $ \mathbb{R}^{3} $, it is assumed that  $ {\bf 0} \in \partial\omega_- \cap \{ x_3=0\}$. In this case, 
  $$ \mathbf{X}_{+}^{\epsilon}=(0, 0, \epsilon). $$ 
We denote $\mathbf{X}_-$ the point in $\partial \omega_- \cap \{ x_3=0\}$ with the same horizontal coordinates as $\mathbf{X}_+^ \epsilon $. Indeed, $\mathbf{X}_-={\bf 0}$. 
\end{enumerate}

\end{definition}

The above constructions can be found in \cite{Shkoller-2019-APAN}, we list them here for the convenience of readers.

Hereafter, $ \mathbf{N} $ and $ \mathbf{N}^{\epsilon} $ are used to denoted the unit outer normal of $ \partial\Omega $ and $ \partial \Omega ^{\epsilon} $, respectively. For each $ \alpha =1,2$ and $\mathbf{x} \in \Gamma^ \epsilon$ ,  $\boldsymbol{\tau}_{\alpha}  ^{\epsilon}(\mathbf{x})$ denotes an orthonormal basis of the $2$-dimensional tangent space to $ \Gamma^\epsilon$ at the point $\mathbf{x}$. Set $ \mathcal{B}_{l}=B ^{+} $ for $ l=1,2,\cdots,K $, and $ \mathcal{B}_{l}=B $ for $ l=K+1,\cdots,L $, where $ B $ and $ B ^{+} $ are as in Section \ref{localchart}. In what follows, we shall recall some facts from \cite{Shkoller-2019-APAN} in terms of the local charts of $ \Omega ^{\epsilon} $.
\begin{lemma}[\mbox{cf.~\cite[Sec. 3 and 5]{Shkoller-2019-APAN}}]\label{lem-te-l-bd}
Let $ \left( \mathcal{U}, \cup _{l=1}^{L}\theta _{l} \right)  $ be the local charts of $ \Omega $ defined in Section \ref{localchart}. Then $ \Omega ^{\epsilon} $ admits local charts $ \left(\mathcal{U}^{\epsilon}, \cup _{l=1}^{L}\theta _{l}^{\epsilon} \right)  $ such that $ \mathop{\mathrm{det}}\nabla \theta _{l}^{\varepsilon}=\tilde{C}_{l} $ for some constants $ \tilde{C}_{l}>0 $, and
\begin{align}\label{bd-te-l}
\displaystyle  \|\nabla \theta _{l}^{\epsilon}\|_{k-1,\mathcal{B}_{l}}  \leq \left( 1+ \frac{4}{h} \right)\|\nabla \theta _{l}\|_{k-1,\mathcal{B}_{l}} 
\end{align}
for $ k \geq 3 $, provided $ \epsilon $ is small enough, where $ h $ is as in \eqref{Defi-OMDA}. Moreover, there are cut-off functions $ \left\{ \xi _{l}^{\epsilon} \right\}_{l=1}^{L} $ with $ \xi _{l}^{\epsilon} $ supported in the image of the charts $ \theta _{l}^{\epsilon} $ such that $ \xi _{l}^{\epsilon}\circ \theta _{l}^{\epsilon}=\xi _{l}\circ \theta _{l} $ and  
\begin{align}\label{alm-unit}
\displaystyle  \sum _{l=1}^{L}\xi _{l}^{\epsilon}(\mathbf{x})\geq 1\ \ \forall \mathbf{x} \in \Omega ^{\epsilon}.
\end{align}
\end{lemma}
It is worth noting that the estimate \eqref{bd-te-l} enables us to derive some elliptic and Sobolev inequalities in the domain $ \Omega ^{\varepsilon} $ with the constants independent of $ \varepsilon $. Clearly, we have  
\begin{lemma}[Estimates for the Stokes problem, \mbox{cf. \cite[Lemma 2]{Shkoller-2019-APAN}}]\label{lem-Stoke}
For integer $k \geq 2  $,  
the following Stokes problem 
\begin{align}
\displaystyle \begin{cases}
  - \Delta u + \nabla p=f, & \mbox{in }\Omega ^{\epsilon},\\
  \mathop{\mathrm{div}}\nolimits u= \psi& \mbox{in }\Omega ^{\epsilon},\\
  u =g & \mbox{on }\Gamma ^{\epsilon}.
\end{cases} 
\end{align}
admits a unique solution $ u \in H ^{k}(\Omega ^{\epsilon}) $ and $ p \in H ^{k-1}(\Omega ^{\epsilon}) $ provided $ f \in H ^{k-2}(\Omega ^{\epsilon}) $, $ \psi \in H ^{k-1}(\Omega ^{\epsilon}) $, and $ g \in H ^{k-1/2}(\Gamma ^{\epsilon})$, and $ \int _{\Omega ^{\epsilon}}\psi(x)\mathrm{d}\mathbf{x}= \int _{\Gamma ^{\epsilon}}g \cdot \mathbf{N} _{\epsilon}\mathrm{d}S $. Moreover, it holds that \begin{align}
\displaystyle \|u\|_{k, \Omega ^{\epsilon}}+ \|\nabla p\| _{k-1,\Omega ^{\epsilon}} \leq C (\|f\| _{k-2,\Omega ^{\epsilon}}+\|\psi\| _{k-1, \Omega ^{\epsilon}}+\| g \| _{k-1/2, \Gamma ^{\epsilon}})
\end{align}
for a positive constant $ C $ depending only on $ \Omega $, but not on $ \epsilon $.

\end{lemma}

\begin{lemma}[\mbox{cf. \cite[Lemma 3]{Shkoller-2019-APAN}}]\label{lem-Sobolev}
It holds that
\begin{align*}
\max _{\mathbf{x}\in \Omega ^{\epsilon}}\vert u (\mathbf{x})\vert \leq C \|u\|_{s,\Omega ^{\epsilon}},\ \  \forall u \in H ^{s}(\Omega ^{\epsilon}),\ \ s >1.5,
\end{align*}
where $C$ is a positive constant depending only on the domain $ \Omega$, but not on $\epsilon$.
\end{lemma}
\begin{lemma}[\mbox{cf. \cite[Lemma 4]{Shkoller-2019-APAN}}]
It holds that,
\begin{align*}
\displaystyle  \max _{\mathbf{x} \in \Gamma ^{\epsilon}}\vert u(x)\vert \leq C \|u\|_{s,\Gamma ^{\epsilon}},\ \ \forall u \in H ^{s}(\Gamma ^{\epsilon}) ,\ \ s>1, 
\end{align*}
where $C$ is a constant depending only on the domain $ \Omega$, but not on $\epsilon$.

\end{lemma}
\begin{lemma}[\mbox{cf. \cite[Lemma 5]{Shkoller-2019-APAN}}]\label{lem-trace}
 There exists a constant $ C>0 $ depending only on the domain $ \Omega $, but not on $\epsilon$, such that for $ s \in (\frac{1}{2},3] $, 
\begin{align*}
\displaystyle \|u\|_{s- \frac{1}{2},\Gamma ^{\epsilon}} \leq C \|u\|_{s,\Omega ^{\epsilon}}, \ \ \forall u \in  H ^{s}(\Omega ^{\epsilon}). 
\end{align*}

\end{lemma}
In the next lemma, we introduce a regularity result for the div-curl system, which will be used later for the construction of the initial magnetic field. 
\begin{lemma}[Estimates for div-curl system, \mbox{cf. \cite[Theorem 1.1]{Cheng-Shkoller-2017-JMFM}}]\label{lem-div-curl}
Suppose that $\Omega\subseteq \mathbb{R}^3$ is a simply connected, bounded $H^{k+1}\,(k>1.5)$-domain. Let $ N $ be  the unit outer normal vector field on $ \partial \Omega $. Assume that a vector field $F\in H^{m-1}(\Omega)$ and a function $g\in H^{m-1}(\Omega)$  ($ 1\leq m \leq k$) satisfy (1) $\operatorname{div} F =0$, (2) $\int_{\Gamma} F \cdot N \mathrm{d}S =0 $ for each connected component $\Gamma$ of $\partial\Omega$, (3) $\int_{\partial\Omega} h \mathrm{d}S =\int_{\Omega} g \mathrm{d}x $ for some function $h\in H^{m-0.5}(\partial\Omega ) $. Then  the system
\begin{subequations}
\label{dc-eq}
\begin{alignat}{2}
\operatorname{curl} v &=F \ \ && \text{ in } \Omega \,,  \\
\operatorname{div} v  &= g \ \ && \text{ in } \Omega  \,, \\
v\cdot N &=h \ \ && \text{ on } \partial\Omega 
\end{alignat}
\end{subequations}
admits a unique solution $v\in H^m(\Omega)$. Moreover, it holds the following estimates: 
\begin{equation}\label{dc-est}
\|v\|_{H^m(\Omega)}\leq C(\vert \partial \Omega \vert ^{k+0.5})[\|F\|_{H^{m-1}(\Omega)}+\|g\|_{H^{m-1}(\Omega)}+\|h\|_{H^{m-0.5}(\Omega)} ],
\end{equation}
where $ \vert \partial \Omega \vert ^{k+0.5} $ represents the $ H ^{k+0.5} $-regularity of the boundary.
\end{lemma}

\subsection{The initial velocity field} 
\label{ssub:subsubsection_name}
As in \cite{Shkoller-2019-APAN}, let $g_0^\epsilon $ be a smooth function defined on $\Gamma ^ \epsilon $ such that
(1) $g_0^ \epsilon = -1$  in a small neighborhood of $X_+^ \epsilon $ on $\partial \omega_+^ \epsilon $
(2) $g_0^ \epsilon =0$ on $\partial \omega_- $ and $\partial \omega^ \epsilon\cap \Gamma^\epsilon  $
(3) $\int_{ \Gamma ^ \epsilon } g_0^ \epsilon \, \mathrm{d} S=0$
(4) $\|g_0^ \epsilon \|_{2.5,\Gamma^ \epsilon } \le m_0 < \infty $, with $m_0$ is independent on $ \epsilon $. Then we can define the initial velocity field $\mathbf{u}_0^ \epsilon $ at $t=0$ as the solution to the following Stokes problem:
\begin{subequations}
\label{Stokes2}
\begin{alignat}{2}
- \Delta \mathbf{u}_0^ \epsilon  + \nabla s_0^ \epsilon & = \mathbf{0}  \ && \text{ in } \Omega^ \epsilon\,,  \\
\operatorname{div} \mathbf{u}_0^ \epsilon  &= 0 \ \ && \text{ in } \Omega^ \epsilon  \,,  \\
[\operatorname{Def} \mathbf{u}_0^ \epsilon \cdot \mathbf{N}^ \epsilon ] \cdot \boldsymbol{\tau} _ \alpha ^ \epsilon  &=0 \ \ && \text{ on } \Gamma^ \epsilon \,, \\
\mathbf{u}_0^ \epsilon \cdot \mathbf{N}^ \epsilon  &=g _{0}^{\epsilon} \ \ && \text{ on } \Gamma^ \epsilon \,,
\end{alignat}
\end{subequations}
Thanks to the elliptic theory of the above elliptic system (cf. \cite{Amrouche-stokes-2011}) alongside \eqref{bd-te-l}, we have
\begin{align*}
\displaystyle  \|\mathbf{u}_0^ \epsilon \|_{3, \Omega^ \epsilon } \le C \|g _{0}^ \epsilon \|_{2.5,\Gamma^ \epsilon } \leq Cm _{0},
\end{align*}
where the constant $C>0$ is independent of $\epsilon$.

\subsection{The initial magnetic field} 
\label{sec:the_initial_magnetic_field}
We first define a vector field $\mathbf{F}_0$ in $\overline{\Omega}$ such that (1) $\mathbf{F}_0 \in (C^\infty(\overline{\Omega}))^3$, (2) $\operatorname{div} \mathbf{F}_0 =0 $ in $\Omega$, (3) $\int_{\partial \omega \cap \Gamma} \mathbf{F}_0 \cdot \mathbf{N} \mathrm{d}S =0 $ and $\int_{\Gamma} \mathbf{F}_0 \cdot \mathbf{N} \mathrm{d}S =0 $, (4)$\mathbf{F}_0^3=0$ and $\partial_{x_3} \mathbf{F}_0 =0$ on $\partial \omega - (\overline{\omega} \cap \Gamma)$, (5)$\| \mathbf{F}_0 \|_{H^1(\Omega)} <\infty$. 
And then for each $0<\epsilon \ll 1$, we define a vector field $\mathbf{F}_0^\epsilon$ in $\Omega^\epsilon $. For $\mathbf{x} \in \omega^\epsilon \cup (\partial \omega \cap \Gamma)$, 
\begin{align*}
\displaystyle ( F_0^\epsilon)^j (x_1,x_2,x_3)=\begin{cases}
  \displaystyle (F_0)^j(x_1,x_{2},2+h-\frac{h(2+h-x_3)}{(h+1-\epsilon )}),& j = 1,2, \\[3mm]
  \displaystyle \frac{h+1-\epsilon}{h }F_0^j(x_1,x_{2},2+h-\frac{h(2+h-x_3)}{(h+1-\epsilon )}),& j =3.
\end{cases}
\end{align*}
For $\mathbf{x} \in \Omega-\omega^\epsilon \cup (\partial \omega \cap \Gamma)$, we define
\begin{align*}
\displaystyle \mathbf{F}_0^\epsilon(x_1,x_2,x_3)=\begin{cases}
  \displaystyle \mathbf{F}_0(x_1,x_{2},x_3+1-\epsilon),& \mathbf{x} \in \omega_+^\epsilon, \\[2mm]
  \displaystyle \mathbf{F}_0(x_1,x_2,x_3),& else. \\
\end{cases}
\end{align*}
By the definition of $\mathbf{F}_0^\epsilon$, we have $\operatorname{div} \mathbf{F}_0^\epsilon =0 $ in $\Omega$. The condition (4) for $\mathbf{F}_0$ ensures that $\mathbf{F}_0^\epsilon$ is still in $H^1(\Omega)$, and (3) ensures $\int_{\Gamma} \mathbf{F}_0^\epsilon \cdot \mathbf{N}^{\epsilon} \mathrm{d}S =0 $. Then we can define $\mathbf{H}_0^\epsilon$ by the following div-curl system:
\begin{subequations}
  \label{h0}
\begin{alignat}{2}
& \mathop{\rm curl}\nolimits  \mathbf{H}_{0}^{\epsilon}= \mathbf{F}_0^\epsilon \ \  \ &&\text{in} \ \ \Omega^ \epsilon  ,\\
& \mathop{\mathrm{div}}\nolimits \mathbf{H}_{0}^{\epsilon} = 0 \ \  \ &&\text{in} \ \ \Omega^ \epsilon  ,\\
&\mathbf{H} _{0}^{\epsilon}\cdot \mathbf{N}^{\epsilon} =0 \ \ &&\text{on} \ \ \Gamma^ \epsilon   .
\end{alignat}
\end{subequations}
According to Lemma \ref{lem-div-curl}, the initial magnetic filed $ \mathbf{H}_{0}^{\epsilon} $ does exist, and satisfies the following estimates: 
\begin{equation}\label{h0-est}
\|\mathbf{H}_0^ \epsilon \|_{2, \Omega^ \epsilon } \le C(\vert \partial\Omega^\epsilon  \vert_{H^{2.5}})  \| \mathbf{F}_0 ^\epsilon \|_{1,\Omega^ \epsilon}^{2} \,
\end{equation} 
Recalling the definition of $\theta_l(l=1,...,K)$ and $\Omega^\epsilon$ as well as \eqref{bd-te-l}, we know that there is a constant $C>0$ independent of $\epsilon$ such that $\vert \partial\Omega^\epsilon  \vert_{H^{2.5}} \leq C \vert \partial\Omega \vert_{H^{2.5}}$. Furthermore, if we assume $ \|\mathbf{F} _{0}\|_{1,\Omega } \leq m _{1} $ for some constant $ m _{1}>0 $ independent of $ \epsilon $. Then we have  $\| \mathbf{F}_0 ^\epsilon \|_{1,\Omega^ \epsilon}^{2} \leq C \| \mathbf{F}_0 \|_{1,\Omega}^{2} \leq Cm _{1} $, where the constant $ C> 0$ is independent of $ \epsilon $. This means that we can construct a family of initial magnetic fields $ \mathbf{H}_{0}^{\epsilon} $ such that 
\begin{gather*}
\displaystyle \|\mathbf{H}_{0}^{\epsilon}\|_{2, \Omega ^{\epsilon}} \leq C m _{1}. 
\end{gather*}
\begin{remark}
The admissible set of $ \mathbf{F}_{0} $ is not empty. For example, one may take $ \mathbf{F}_{0}=(x _{2},-x _{1},0) $.
\end{remark}
\begin{remark}
In the above construction on the initial magnetic field, we do not require that the initial magnetic field vanishes at the boundary. And we emphasize that one can construct the initial magnetic field with zero boundary value in a much simpler way. Indeed, we can define $ \mathbf{H}_{0}^{\epsilon} $ by solving the following Stokes problem:
\begin{align*}
\displaystyle \begin{cases}
 \displaystyle - \Delta \mathbf{H} _{0}^{\epsilon}+\nabla \chi _{0}^{\epsilon} =\mathbf{G} _{0}^{\epsilon} & \mbox{in }\Omega_{\epsilon},\\
  \displaystyle  \mathop{\mathrm{div}}\nolimits \mathbf{H} _{0}^{\epsilon}=0 & \mbox{in }\Omega_{\epsilon},\\ 
  \displaystyle\mathbf{H} _{0}^{\epsilon}=0 & \mbox{on }\Gamma_{\epsilon},
\end{cases} 
\end{align*}
where $ \mathbf{G} _{0}^{\epsilon} \in L ^{2}(\Omega ^{\epsilon}) $ with $ \| \mathbf{G} _{0}^{\epsilon}\| _{0,\Omega ^{\epsilon}} \leq C$ for some constant $ C>0 $ independent of $ \epsilon $. This kind of function $\mathbf{G} _{0}^{\epsilon} $ can be easily constructed by virtue of the coordinate charts $ \theta _{l} $ and some fixed vector-valued function $ \mathbf{G} _{0}\in L ^{2}(B)$. However, according to \cite{Gu-Wangyanjin-JMPA}, if the initial magnetic field vanishes at the boundary, the magnetic field will always vanish at the boundary in the lifespan of the solution. Therefore, to some extent, our construction on the initial magnetic field is quite general.
\end{remark}


\subsection{The initial pressure function}
With the initial velocity field $\mathbf{u}_{0}^{ \epsilon}$ and initial magnetic field $\mathbf{H} _{0}^{\epsilon}$, we define the initial total pressure $ P _{0}^{\epsilon} $ through the following elliptic problem:
\begin{subequations}
  \label{p0}
\begin{alignat}{2}
- \Delta P_0^ \epsilon   &=  (u_0^ \epsilon )^i,_j (u_0^ \epsilon )^j,_i  -\frac{1}{4\pi}H _{0}^{i}, _{j} H _{0} ^{j}, _{i} \ \  \ &&\text{in} \ \ \Omega^ \epsilon  \,,\\
 P_0^ \epsilon  &=  \mathbf{N}^{\epsilon}  \cdot \left[ \nu   \operatorname{Def} \mathbf{u}_0^ \epsilon  \cdot \mathbf{N}^{\epsilon}  \right] \ \ &&\text{on} \ \ \Gamma^ \epsilon   \,.
\end{alignat}
\end{subequations}
Then by the standard regularity theory of elliptic equations along with the estimate \eqref{bd-te-l}, we get the following $ \epsilon $-independent elliptic estimate:
\begin{equation}\label{p0-est}
\|P_0^ \epsilon \|_{2, \Omega^ \epsilon } \le C \left[ \| \mathbf{u}_0 ^\epsilon \|_{3,\Omega^ \epsilon}+ \|\mathbf{u}_0^ \epsilon \|^2_{3,\Omega^ \epsilon }+\| \mathbf{H} _{0}^{\epsilon}\|_{3, \Omega ^{\epsilon}}^{2} \right]  \le C \,
\mathcal{P} ( m_0,m _{1}) ,
\end{equation}
where we use $\mathcal{P} $ to denote a generic polynomial function that depends only on $\Omega$.

\vspace{6mm}

\section{\emph{A priori} estimates} 
\label{sec-a_priori-esti}


\subsection{Reformulation of the problem in Lagrangian coordinates}  \label{sec::lagrangian}
Now let us consider the problem \eqref{MHDe}--\eqref{free-bd} supplemented with the initial data $ (\Omega ^{\epsilon},\mathbf{u}_{0}^{\epsilon},\mathbf{H}_{0}^{\epsilon}) $, and denote the corresponding solution by $ (\Omega ^{\epsilon}(t),\tilde{\mathbf{u}}^{\epsilon},\tilde{\mathbf{H}}^{\epsilon} ) $. We utilize the Lagrangian flow map to transform the problem on $\Omega ^{\epsilon}(t)$ to one on the initial domain $\Omega^{\epsilon}$. Let $\boldsymbol{\eta} ^{\epsilon}(\cdot,t): \Omega^ {\epsilon } \longrightarrow \Omega ^{\epsilon}(t)$ be flow map such that 
\begin{subequations}
\label{eta}
\begin{alignat}{3}
\p_{t} \boldsymbol{\eta}^{\epsilon}(x,t)&= \tilde{\mathbf{u}}^{\epsilon}( \boldsymbol{\eta}^{\epsilon}(\mathbf{x},t),t) , \ \  &\mathbf{x}& \in \Omega^{\epsilon} , t > 0,\\
\boldsymbol{\eta}^{\epsilon}(\mathbf{x},0)&=\mathbf{x},\ \  &\mathbf{x} &\in \Omega^{\epsilon}.
\end{alignat}
\end{subequations}
Since $ \mathop{\mathrm{div}}\nolimits \tilde{\mathbf{u}}^{\epsilon} =0$, we have $ \mathop{\mathrm{det}}\nabla \boldsymbol{\eta}^{\epsilon}=1 $. Then we know that $\boldsymbol{\eta}^{\epsilon}(\cdot, t)$ is a diffeomorphism from $\Omega^{\epsilon}$ to $\Omega ^{\epsilon}(t)$. Also, because of (\ref{MHDe}f), the boundary of $\Omega^{\epsilon}$ is transformed to the boundary of $\Omega^\epsilon (t)$, i.e. $$
\Gamma ^{\epsilon}(t):=\partial \Omega ^{\epsilon}(t) = \boldsymbol{\eta} ^{\epsilon} ( \Gamma^ \epsilon , t) \,.
$$
Next, we define
\begin{align*}
\mathbf{v}^{\epsilon} &= \tilde{\mathbf{u}}^{\epsilon} \circ \boldsymbol{\eta}^{\epsilon}   \text{ (Lagrangian velocity)},  \\
\mathbf{H}^{\epsilon}&=\tilde{\mathbf{H}}^{\epsilon}  \circ \boldsymbol{\eta}^{\epsilon},  \text{ (Lagrangian magnetic field)}\\
Q ^{\epsilon}&=P^\epsilon \circ \boldsymbol{\eta}^{\epsilon}   \text{ (Lagrangian pressure)}, \\
A ^{\epsilon} &= [ \nabla  \boldsymbol{\eta}^{\epsilon}]^{-1}  \text{ (inverse of the deformation tensor)}\,, \\
g_{ \alpha \beta }^{\epsilon} &=  \boldsymbol{\eta}^{\epsilon}, _ \alpha \cdot \boldsymbol{\eta}^{\epsilon},_\beta \ \ \alpha ,\beta =1,2 \text{ (induced metric on $\Gamma^\epsilon$)}, \\
\mathfrak{g}^{\epsilon}  & = \det( g_{ \alpha \beta }^{\epsilon}) .
\end{align*}
The Lagrangian analogue of some of the fundamental differential operators can also be defined. Precisely, let $F$ be a vector function in $\mathbb{R}^3$, then
\begin{align*}
\operatorname{div} _{\boldsymbol{\eta}^{\epsilon}} F &=  (\operatorname{div} (F \circ (\boldsymbol{\eta}^{\epsilon})^{-1})) \circ \boldsymbol{\eta} = F^i,_j (A ^{\epsilon})^j_i  , \\
\operatorname{curl} _{\boldsymbol{\eta}^{\epsilon}} F &=  (\operatorname{curl} (F \circ (\boldsymbol{\eta}^{\epsilon})^{-1})) \circ \boldsymbol{\eta} \text{ or }  [ \operatorname{curl} _ {\boldsymbol{\eta}^{\epsilon}} F]_i = \varepsilon_{ i jk}  F^k,_r (A ^{\epsilon})^r_j ,  \\
\operatorname{Def} _{\boldsymbol{\eta}^{\epsilon}} F &=  (\operatorname{Def} (F \circ (\boldsymbol{\eta}^{\epsilon})^{-1})) \circ \boldsymbol{\eta}^{\epsilon}  \text{ or } [\operatorname{Def} _ {\boldsymbol{\eta}^{\epsilon}} F]^i_j =  F^i,_r(A ^{\epsilon})^r_j + F^j,_r(A ^{\epsilon})^r_i , \\
\Delta _{\boldsymbol{\eta}^{\epsilon}} F &=  (\Delta (F \circ (\boldsymbol{\eta}^{\epsilon})^{-1})) \circ \boldsymbol{\eta}^{\epsilon} = ( (A ^{\epsilon})^j_r (A ^{\epsilon})^k_r F,_k),_j .
\end{align*}
Hereafter, we drop the superscript $ \epsilon $ of $ (\boldsymbol{\eta}^{\epsilon},\mathbf{v}^{\epsilon},\mathbf{H}^{\epsilon},Q ^{\epsilon},\mathbf{u}_{0}^{\epsilon},\mathbf{H}_{0}^{\epsilon},A ^{\epsilon}) $ for simplicity. Then the Lagrangian version of the system \eqref{MHDe} is given on
the initial domain $\Omega^ \epsilon $ by
\begin{subequations}
\label{mhd-lagran}
\begin{alignat}{2}
\boldsymbol{\eta}( \cdot ,t) & = \mathbf{e} + \int_0^t \mathbf{v}(\cdot ,s) \mathrm{d}s  \ && \text{ in } \Omega^ \epsilon  \times [0,T] ,  \\
 (\partial _{t}\mathbf{v} + A^T \nabla Q   - \nu \Delta _ {\boldsymbol{\eta}} \mathbf{v} )^i &=  \frac{1}{4 \pi} H ^{i}A _{i}^{j}\mathbf{H}, _{j}\ \ && \text{ in } \Omega^ \epsilon  \times (0,T] , \label{v-eq} \\
 \displaystyle \mathbf{H} _{t}  &=H ^{i}A _{i}^{k}\mathbf{v}, _{k}  \ \ && \text{ in } \Omega^\epsilon  \times [0,T] ,\label{B-eq}\\
    \displaystyle\operatorname{div} _{\boldsymbol{\eta}} \mathbf{v} &=0 \ \ && \text{ in } \Omega^\epsilon  \times [0,T] ,  \label{mhd-lagran-c}\\
   \displaystyle\operatorname{div} _{\boldsymbol{\eta}} \mathbf{H} &=0 \ \ && \text{ in } \Omega^\epsilon  \times [0,T] ,  \label{mhd-lagran-B}\\
\nu \operatorname{Def} _ {\boldsymbol{\eta}} \mathbf{v} \cdot \mathbf{n}^{\epsilon} - Q\mathbf{n}^{\epsilon} &=0,\ \ \mathbf{H} \cdot \mathbf{n}^{\epsilon}=0  && \text{ on } \Gamma^ \epsilon  \times [0,T] ,  \label{mhd-lagran-d}\\
(\boldsymbol{\eta} ,\mathbf{v},\mathbf{H})  &=(\mathbf{e},\mathbf{u}_0, \mathbf{H}_{0}) \ \  \ \ && \text{ in } \Omega^\epsilon  \times \{t=0\} ,
\end{alignat}
\end{subequations}
where $\mathbf{e}$ is defined by $\mathbf{e}(\mathbf{x})=\mathbf{x}$ for all $ \mathbf{x} \in \Omega ^{\epsilon}  $. As before, let $\mathbf{N}^ \epsilon$ denote the exterior normal vector to the boundary of $\Omega^ \epsilon $, then we have
$$
\mathbf{n} ^{\epsilon} =\frac{ (A ^{\epsilon})^T \mathbf{N}^{\epsilon}}{\vert  (A ^{\epsilon})^T \mathbf{N}^{\epsilon}\vert} .
$$
Hereafter, we will also drop the superscript $ \epsilon $ of $ \mathbf{n}^{\epsilon} $ for simplicity. The local-in-time existence and uniqueness of solutions to problem \eqref{mhd-lagran} can be obtained by the arguments as in \cite{Solonnikov-1977},  \cite{Solonnikov-1992} and \cite{Guo-Tice-2013-local} with slight modification since the magnetic field can be solved in terms of the Jacobian of the flow map and the initial data.  We shall show that both the \emph{a priori} estimates and
the time of existence for solutions are independent of the distance $ \epsilon  >0$ between the falling dinosaur head $\mathbf{X}_+^ \epsilon $ and the flat trough
$ \partial \omega_- \cap \{ x_3=0\}$ (see Figure \ref{Fig-dinau}(b)). To do so, we shall further reformulate the problem \eqref{mhd-lagran} with the boundary local charts $ \left\{ \theta _{l}^{\epsilon} \right\}_{l=1}^{K} $. With $\Omega^ \epsilon  $ and $\theta_l ^{\epsilon}$ defined in Section \ref{sec-initial_data}, we define $\boldsymbol{\eta}_l ^{\epsilon}(t): B^+ \to  \Omega(t) \ \text{ for } \ \ (l=1,...,K )$ by:
$$
\boldsymbol{\eta}_l ^{\epsilon} = \boldsymbol{\eta}^{\epsilon} \circ \theta_l ^{\epsilon}\ \text{ for } \ \ l=1,...,K \,.
$$
We set $\mathbf{v}_l ^{\epsilon} = \tilde{\mathbf{u}} \circ \boldsymbol{\eta}_l ^{\epsilon}$, $Q_l ^{\epsilon} = P ^{\epsilon} \circ \boldsymbol{\eta}_l ^{\epsilon}$, $ \mathbf{H}^\epsilon _{l}= \tilde{\mathbf{H}} \circ \boldsymbol{\eta}_{l}^{\epsilon} $ and $A_l ^{\epsilon}= [ \nabla \boldsymbol{\eta}_l ^{\epsilon} ] ^{-1} $, $J_l ^{\epsilon} = C_l$ (where $C_l>0$ is a constant, and $a_l ^{\epsilon} = J_l ^{\epsilon} A_l ^{\epsilon}$.
 The unit normal
$\mathbf{n}_l ^{\epsilon}$ is defined as $ \mathfrak{g}  ^ {-\frac{1}{2}} \frac{\p \boldsymbol{\eta}_l ^{\epsilon}}{\p y_1} \times \frac{\p \boldsymbol{\eta}_l ^{\epsilon}}{\p y_2} $. It then follows from \eqref{mhd-lagran} that for $l=1,...,K$~(after dropping the superscript $ \epsilon $),
\begin{subequations}
\label{localNS}
\begin{alignat}{2}
\boldsymbol{\eta}_l(\cdot, t) &= \theta_l + \int_0^t \mathbf{v}_l(\cdot,s) \mathrm{d}s\ \ && \text{ in } B^+ \times [0,T] , \label{localNS.a0} \\
\p_t  \mathbf{v}_l  + A_l^T \nabla Q_l   &=  \nu\Delta _ {\boldsymbol{\eta}_l} \mathbf{v}_l +  \frac{1}{4 \pi} \mathbf{H} _{l}\cdot \nabla _{\boldsymbol{\eta} _{l}} \mathbf{H} _{l} \ \ && \text{ in } B^+ \times (0,T] , \label{localMHD-a} \\
\displaystyle\partial _{t} \mathbf{H}_{l}&=\mathbf{H}_{l} ^{k}(A _{l}) _{k}^{\lambda}\mathbf{v}_{l} , _{\lambda}  \ \ && \text{ in } B ^{+}  \times [0,T] \,,\label{Bl-eq}\\
  \displaystyle\operatorname{div} _{\boldsymbol{\eta}_l} \mathbf{v}_l &=0 \ \ && \text{ in } B^+ \times [0,T] \,,\label{localNS.b}  \\
\operatorname{div} _{\boldsymbol{\eta}_l} \mathbf{H}_l &=0 \ \ && \text{ in } B^+ \times [0,T] ,\label{localMHD-B}  \\
\nu \operatorname{Def} _ {\boldsymbol{\eta}_l} \mathbf{v}_l \cdot \mathbf{n}_l - Q_l  \mathbf{n}_l &=0 \ \ && \text{ on } B^0 \times [0,T] ,\label{localNS.c}  \\
\displaystyle \mathbf{H} _{l}\cdot \mathbf{n}_{l}&=0 \ \ && \text{ on } B^0 \times [0,T] ,\label{localMHD-B-bdy}\\
(\boldsymbol{\eta}_l,\mathbf{v}_l,\mathbf{H} _{l})  &=(\theta_l,\mathbf{u}_0 \circ \theta_l, \mathbf{H} _{0}\circ \theta _{l} ) \ \  \ \ && \text{ in } B^+ \times \{t=0\} . \label{localNS.e}
\end{alignat}
\end{subequations}
\subsection{\emph{A priori} estimates} 
\label{sub:em}
In this section, we will establish some \emph{a priori} estimates for the solutions which enable us to show the existence of finite-time splash singularity. From now on, we assume that $ (\boldsymbol{\eta},\mathbf{v},\mathbf{H}, Q) $ is a smooth solution to the problem \eqref{mhd-lagran} on $ [0,T] $ for some $ T>0 $. For any $ t \in [0,T] $, define 
\begin{align*}
\mathcal{E}^\epsilon (t) & =  1+ \| \boldsymbol{\eta}( \cdot ,t)\|_{3, \Omega^ \epsilon  }^2 +  \| \mathbf{v} ( \cdot ,t)\|_{2, \Omega^ \epsilon  }^2 +\|\mathbf{H}\| _{2, \Omega ^{\epsilon}} ^{2}
+ \int_0^t \| \mathbf{v} ( \cdot ,s)\|_{3, \Omega^ \epsilon  }^2 \mathrm{d} s
+ \int_0^t \| Q( \cdot ,s)\|_{2, \Omega^ \epsilon  }^2 \mathrm{d}s  \\
& \qquad +   \| \mathbf{v}_t( \cdot ,t)\|_{0, \Omega^ \epsilon  }^2 +\|\mathbf{H} _{t}\| _{0,\Omega ^{\epsilon}}^{2} 
+ \int_0^t \| \mathbf{v}_s ( \cdot ,s)\|_{1, \Omega^ \epsilon  }^2 \mathrm{d}s,
\end{align*}
and set  $\mathcal{M}_0 = \mathcal{P} ( \mathcal{E}^ \epsilon (0))$, where $\mathcal{P} $ denotes a generic polynomial whose coefficients depend only on $\Omega$. Clearly, $ \mathcal{M}_{0} $ is a positive constant independent of $ \epsilon $. Here is a crucial proposition of this section.
\begin{proposition}\label{prop-key}
Assuming that $\Gamma(t)$ does not self-intersect, independent of $ \epsilon  >0$, there exists a time $T>0$ and a constant $C>0$ such that the smooth solution  $ (\boldsymbol{\eta},\mathbf{v},\mathbf{H}, Q) $ to the problem \eqref{mhd-lagran} on $ [0,T] $ satisfies the a priori estimate:
\begin{align}
 \max_{t\in [0,T]} \mathcal{E}^ \epsilon (t)    \le C\,  \mathcal{M}_0   .   \label{main-est}
\end{align}
\end{proposition}
\begin{proof}
 The proof of Proposition \ref{prop-key} is made up of Lemmas \ref{7.1}--\ref{lem-verify}.
\end{proof}

To prove this proposition, we first make an \emph{a priori} assumption: for a constant $0<\vartheta \ll 1$, 
we suppose that there is a sufficiently small $T$ independent of $ \epsilon $ such that for any $t \in [0,T]$,
\begin{equation}\label{basic}
\sup_{t \in [0,T]} \| \nabla \boldsymbol{\eta} (\cdot , t)- \mathbb{I}_{3} \|_{ L^\infty(\Omega^ \epsilon )} \le \vartheta^{10}.
\end{equation}
As a direct consequence of \eqref{basic}, we have the following lemma.
\begin{lemma}\label{7.1}
Under the \emph{a priori} assumption \eqref{basic}, it holds for sufficiently small $ \vartheta $ that
\begin{equation}\label{est-A}
\sup_{t \in [0,T]} \|A ( \cdot , t) - \mathbb{I}_{3}\|_{L^\infty (\Omega^ \epsilon )} +   \|A A^T( \cdot , t) - \mathbb{I}_{3} \|_{L^\infty (\Omega^ \epsilon )}
\le  \vartheta.
\end{equation}
\end{lemma}
\begin{proof}
 From \eqref{basic}, we get  
 \begin{align}\label{na-eta-bd}
 \displaystyle \|\nabla \boldsymbol{\eta}\|_{L ^{\infty}(\Omega ^{\epsilon})} \leq C  
 \end{align}
 for some constant $ C>0 $ independent of $ \epsilon $, and 
 \begin{align}\label{na-eta-elements}
 \displaystyle \|\eta ^{i},_{j}- \delta ^{ij}\| _{L ^{\infty}(\Omega ^{\epsilon})} \leq \vartheta ^{10}.
 \end{align}
 Notice that 
 \begin{align}\label{A-express}
 \displaystyle  A=\left[\begin{matrix} {\boldsymbol{\eta},_2 \times \boldsymbol{\eta},_3}\\ {\boldsymbol{\eta},_3 \times \boldsymbol{\eta},_1}\\ {\boldsymbol{\eta},_1 \times \boldsymbol{\eta},_2} \end{matrix} \right]
 \end{align}
 due to $ \mathop{\mathrm{det}}\nabla \boldsymbol{\eta}=1 $. Then we get by virtue of \eqref{na-eta-bd} and \eqref{na-eta-elements} that 
 \begin{align}\label{A-bd}
  \displaystyle \|A- \mathbb{I}_{3}\|_{L ^{\infty}(\Omega ^{\epsilon})}=\|A ^{T}- \mathbb{I}_{3}\| _{L ^{\infty}(\Omega ^{\epsilon})}\leq C \vartheta ^{10},   
  \end{align}
   where $ C>0 $ is a constant independent of $ \epsilon $ and $ \vartheta $. This along with the fact 
   \begin{align*}
   \displaystyle  AA ^{T}- \mathbb{I}_{3}=A (A ^{T}-\mathbb{I}_{3})+(A- \mathbb{I}_{3})\mathbb{I}_{3} 
   \end{align*}
   further implies that 
   \begin{align}\label{A-AT-bd}
   \displaystyle  \|AA ^{T}- \mathbb{I}_{3}\| _{L ^{\infty}(\Omega ^{\epsilon})}\leq C \vartheta ^{10}.
   \end{align}
Combining \eqref{A-bd} and \eqref{A-AT-bd}, and setting $ \vartheta $ suitably small, we then get \eqref{est-A}, and thus finish the proof of the present lemma. 
 \end{proof}

In the following lemma, we shall derive the boundary regularity of the velocity field. To be clear, throughout this section, we denote by $ \|\cdot\|_{X,\Omega ^{\epsilon}} $ the $ X $-norm of the function defined on the initial domain, and denote by $ \|\cdot\|_{Y, B ^{+}} $ the $ Y $-norm of the function defined on $ B ^{+} $. In particular, $ \|\mathbf{G}\|_{Y,B ^{+}} $ should be regarded as $ \|\mathbf{G}_{l}\|_{Y,B ^{+}} $ for any vector-valued function $ \mathbf{G} _{l} $ defined in $ B ^{+} $.
\begin{lemma}\label{lem-bdy-regul}
For any positive number $\delta $, there exists a constant $C_{\delta}$ independent of $\epsilon$,   such that
\begin{equation}\label{cs6}
 \int_0^T   \|   \mathbf{v}(\cdot , t)  \|^2_{2.5, \Gamma^ \epsilon } \mathrm{d}t \leq 
\mathcal{M}_0 +  C_ \delta  T P( \sup_{t \in [0,T]} \mathcal{E}^ \epsilon (t)) + \delta   \sup_{t \in [0,T]} \mathcal{E}^ \epsilon (t)
\end{equation}
\end{lemma}

\begin{proof}
Recall that $ \xi _{l} $ is the cut-off function associated to $ \theta _{l}(B ^{+}) $ for $ l=1,\cdots,K $. And recall the definition of $ \zeta _{l}^{\epsilon}\,(\mbox{drop the superscript } \epsilon \mbox{ later}) $ in Lemma \ref{lem-te-l-bd}. By equation \eqref{localMHD-a}, we get
\begin{equation}\label{cs0}
\int_{B^+} ( \zeta_l \bar \p^2 [ \p_t  \mathbf{v}_l  + A_l^T \nabla Q_l   - \nu\Delta _ {\boldsymbol{\eta}_l} \mathbf{v}_l +  \frac{1}{4 \pi} \mathbf{H} _{l}\cdot \nabla _{\boldsymbol{\eta} _{l}} \mathbf{H} _{l}] \cdot  \zeta_l \bar  \partial^2 \mathbf{v}_l \mathrm{d}\mathbf{y}  =0 \,.
\end{equation}
We fix $l \in \{1,...,K\}$ and drop the subscript to simplify the notation. Then it follows that
\begin{align}\label{split-bdy-esti}
&\int_{B^+} \zeta ^2 \bar \partial^2 v_t^i \,  \bar \partial^2 v^i \, \mathrm{d} \mathbf{y} - \int_{B^+}  \zeta ^2 \bar \partial^2 [ A^k_s A^j_s v ^i,_j],_k \, \bar \partial^2 v^i \, \mathrm{d} \mathbf{y} + \int_{B^+} \zeta ^2 \bar \partial^2
[A^k_i Q],_k \, \bar \partial^2 v^i \, \mathrm{d}\mathbf{y}
 \nonumber \\
 &~\displaystyle =  \frac{1}{4 \pi} \int _{B ^{+}}\zeta ^{2} H^{i}(A _{l}) _{i}^{s}H^{j},_{s} \bar \partial^2 v ^{j} \mathrm{d}\mathbf{y},
\end{align}
where we have used the Piola identity, i.e., $ (J _{l} (A _{l})_{i}^{k}),_{k}=0 $ and the fact $ J _{l}=\tilde{C} _{l} $ for some constant $ \tilde{C}_{l} $. By using integration by parts and the boundary condition (\ref{localNS}d), we get from \eqref{split-bdy-esti} that 
\begin{align}
 &{\frac{1}{2}} \frac{\mathrm{d}}{\mathrm{d}t} \| \zeta \bar \p^2 \mathbf{v}(\cdot , t)\|^2_{0,B^+}
+  \| \zeta \bar\p^2  \nabla \mathbf{v}(\cdot, t)  \|^2_{0,B^+}
 \nonumber \\ 
 & ~\displaystyle 
= -   \int_{B^+}  \bar \p^2 [A^k_i  Q]\, [ \zeta ^2 \bar \partial^2v^i],_k \mathrm{d}\mathbf{y} \nonumber -  \int_{B^+}  \bar \p^2 [ (A^k_{\lambda} A^j_{\lambda} -\delta^{kj}) v^i,_j]\, [ \zeta ^2 \bar \partial^2v^i,_k] \mathrm{d}\mathbf{y}
 \nonumber \\ 
 & \displaystyle ~ \quad
-2 \int _{B ^{+}}  \bar{\partial}^{2}(A^k_{\lambda} A^j_{\lambda} v ^{i}, _{j}) \zeta \zeta, _{k}\bar \partial^2v^i\mathrm{d}\mathbf{y}
 +  \frac{1}{4 \pi} \int _{B ^{+}}\zeta ^{2}\bar{\partial}^{2} [H^{i}A  _{i}^{\lambda}H^{j} ,_{\lambda}] \bar \partial^2 v ^{j} \mathrm{d}\mathbf{y} .  \label{cs2}
\end{align}
Notice that
\begin{align*}
\displaystyle \int_{B^+}  \zeta \bar \p^2 [ \partial _{t}\mathbf{H}-H ^{k}A  _{k}^{\lambda}\mathbf{v} , _{\lambda}]\cdot \ \frac{1}{4 \pi}\zeta\bar  \partial^2 \mathbf{H} \mathrm{d}\mathbf{y} =0 .
\end{align*}
Then we have
\begin{align}\label{H-iden-bdy}
\displaystyle  \frac{1}{8 \pi}\frac{\mathrm{d}}{\mathrm{d}t}\|\zeta \bar{\partial}^{2} \mathbf{H} \|_{0, B ^{+}}^{2}= \frac{1}{4 \pi}\int _{B ^{+}}\zeta ^{2}\bar{\partial}^{2}[ H^{k}A  _{k}^{\lambda}v  ^{i}, _{\lambda}] \bar{\partial}^{2} H^{i}\mathrm{d}\mathbf{y}.
\end{align}
Integrating \eqref{cs2} and \eqref{H-iden-bdy} over the time interval $[0,t]\subset [0,T]$, we get 
\begin{align}\label{cs7}
 &\frac{1}{2}  \| \zeta \bar \p^2 \mathbf{v}(\cdot , t)\|^2_{0,B^+} +\frac{1}{8 \pi}\|\zeta \bar{\partial}^{2} \mathbf{H} \|_{0, B ^{+}}^{2}+ \int_0^t   \| \zeta \bar{\partial}^{2}  \nabla \mathbf{v}(\cdot,s)  \|^2_{0,B^+}\mathrm{d}s 
  \nonumber \\ 
  &~\displaystyle \leq  M_0 
 +  \mathcal{R} _1 
 + \mathcal{R} _2  +  \mathcal{R} _3  +  \mathcal{R}_{4}  + \mathcal{R}_{5} , 
\end{align}
where $ \mathcal{R}_{i}\,(i=1,2,\cdots,5) $ are given by 
\begin{align*}
\mathcal{R} _1 & =  \int_0^t \int_{B^+} \left \vert  \bar \p^2 [A^k_i  Q]\, [ \zeta ^2 \bar\p^2v^i],_k\right\vert \mathrm{d}\mathbf{y} \mathrm{d}s , \\
\mathcal{R} _2 & =\int_0^t  \int_{B^+}\left \vert \bar \p^2 [ (A^k_{\lambda} A^j_{\lambda} - \delta^{kj}) v^i,_j]\, [ \zeta ^2 \bar \partial^2v^i],_k \right\vert \mathrm{d}\mathbf{y} \mathrm{d}s, \\
\mathcal{R} _3 & =-2 \int_0^t\int _{B ^{+}}  \bar{\partial}^{2}(A^k_{\lambda} A^j_{\lambda} v ^{i}, _{j}) \zeta \zeta, _{k}\bar \partial^2v^i\mathrm{d}\mathbf{y} \mathrm{d}s
 ,\\
 \displaystyle \mathcal{R}_{4}&= \frac{1}{4 \pi} \int_0^t\int _{B ^{+}}\zeta ^{2}\bar{\partial}^{2} [H^{i}A _{i}^{\lambda} H ^{j},_{\lambda}] \bar \partial^2 v ^{j} \mathrm{d}\mathbf{y}\mathrm{d}s, \\
 \displaystyle \mathcal{R}_{5}&=\frac{1}{4 \pi}\int_0^t\int _{B ^{+}}\zeta ^{2}\bar{\partial}^{2}[ H ^{k}A  _{k}^{\lambda}v  ^{i}, _{\lambda}] \bar{\partial}^{2} H ^{i}\mathrm{d}\mathbf{y} \mathrm{d}s.
\end{align*}
Using \eqref{bd-te-l}, \eqref{na-eta-bd}, the Sobolev embedding theorem and Lemma \ref{lem-Sobolev}, we estimate $ \mathcal{R} _1$ as 
\begin{align}\label{cal-R-1}
\mathcal{R}_1 & \leq C\int _{0}^{T}\int _{B ^{+}}\left( \vert  A _{i}^{k}\vert \vert \bar{\partial}^{2}Q\vert+\vert \bar{\partial}^{2}A _{i}^{k}\vert \vert Q\vert+ \vert \bar{\partial}A _{i}^{k}\vert \vert \bar{\partial}Q\vert\right)\left( 2 \vert \zeta\vert \vert  \zeta, _{k}\vert \vert \bar{\partial}^{2}v ^{i}\vert+ \zeta ^{2}\vert \bar{\partial}^{2}v ^{i}, _{k}\vert\right)\mathrm{d}\mathbf{y}\mathrm{d}t 
 \nonumber \\ 
 & \displaystyle \leq \underbrace{C\int _{0}^{T}\int _{B ^{+}}\left( \vert \bar{\partial}^{2}Q\vert+\vert \bar{\partial}^{2}A\vert \vert Q\vert +\vert \bar{\partial}A\vert \vert \bar{\partial}Q\vert\right) \vert \bar{\partial}^{2}v ^{i}, _{k}\vert \mathrm{d}\mathbf{y}\mathrm{d}t }_{\mathcal{R}_{1}^{a}}
  \nonumber \\ 
  & \displaystyle \quad +C \int _{0}^{T}\|A\| _{L ^{\infty}(B ^{+})}\|\bar{\partial}^{2}Q\|_{0,B ^{+}}\|\mathbf{v}\|_{2,B ^{+}}\mathrm{d}t + C \int _{0}^{T}\|Q\|_{L ^{\infty}(B ^{+})}\|\bar{\partial}^{2}A\|_{0,B ^{+}}\|\mathbf{v}\|_{2,B ^{+}}\mathrm{d}t \nonumber \\ 
  & \displaystyle \quad+ \int _{0}^{T}\|\bar{\partial}A\|_{L ^{4}(B ^{+})}\|\bar{\partial}Q\|_{L ^{4}(B ^{+})}\|\mathbf{v}\|_{2, B ^{+}}\mathrm{d}t 
\nonumber \\ 
 & \displaystyle \leq \mathcal{R}_{1}^{a}+\underbrace{
  \int_0^T  \| Q\|_{2, \Omega ^{\epsilon} }   \| A\|_{2, \Omega^ \epsilon} \| \mathbf{v}\|_{2, \Omega ^ \epsilon } \mathrm{d} t }_{\mathcal{R}_{1}^{b}},
\end{align}
where, thanks to the Cauchy-Schwarz inequality, we have for $ \mathcal{R}_{1}^{b} $ that 
\begin{align}
\displaystyle \mathcal{R}_{1}^{b} \leq  \delta \int _{0}^{T}\|Q\|_{2,\Omega ^{\epsilon}}^{2}\mathrm{d}t +C _{\delta}T \mathcal{P}(\sup _{t \in [0,T]}\mathcal{E}(t))
\end{align}
for any $ \delta>0 $, where $ C _{\delta}>0 $ is a constant depending on $ \delta $ (which blows-up as $ \delta \to 0$). To estimate the integral $ \mathcal{R} _1^a$, we use \eqref{localNS.b} to get
$$
v^i,_{ k \alpha \beta } A^k_i = - A^k_i,_{ \alpha  \beta } v^i,_k - A^k_i,_ \beta v^i,_{k \alpha } - A^k_i,_ \alpha  v^i,_{k \beta},
$$
so that the term with third derivatives on $\mathbf{v}$ is converted to a term with third derivatives on $\eta$ plus lower-order terms. Here repeated Greek indices $\alpha, \beta$ are summed from $1$ to $2$. Therefore it follows that for $ \delta >0$, 
\begin{align}
\displaystyle  \mathcal{R} _1^a &\leq C \int _{0}^{T}\int _{\Omega ^{\epsilon}}(\vert Q\vert+\nabla Q\vert+\vert \nabla ^{2}Q\vert)(\vert \nabla ^{2}\boldsymbol{\eta}\vert ^{2} +1) (\vert \nabla \mathbf{v}\vert+\vert \nabla ^{2}\mathbf{v}\vert)  \mathrm{d}\mathbf{x}\mathrm{d}t  
 \nonumber \\ 
 & \displaystyle \quad + C \int _{0}^{T}\int _{\Omega ^{\epsilon}}(\vert Q\vert+\nabla Q\vert+\vert \nabla ^{2}Q\vert)(\vert \nabla ^{2}\boldsymbol{\eta}\vert  +1) \vert \nabla ^{3}\boldsymbol{\eta}\vert \vert \nabla\mathbf{v}\vert\mathrm{d}\mathbf{x}\mathrm{d}t  
 \nonumber \\ 
 & \displaystyle \quad+ C\int _{0}^{T}\int _{\Omega ^{\epsilon}} \vert \nabla ^{3} \boldsymbol{\eta}\vert \vert Q\vert \Big[(\vert \nabla ^{2}\boldsymbol{\eta} \vert  +1)(\vert \nabla \mathbf{v}\vert+ \vert \nabla ^{2}\mathbf{v}\vert)+\vert \nabla ^{3}\boldsymbol{\eta}\vert \vert \nabla \mathbf{v}\vert \Big]\mathrm{d}\mathbf{x}\mathrm{d}t  
  \nonumber \\ 
    & \displaystyle \leq  C\int _{0}^{T}(\|Q\| _{2,\Omega ^{\epsilon}}\| \boldsymbol{\eta}\|_{3, \Omega ^{\epsilon}}\|\mathbf{v}\|_{2, \Omega ^{\epsilon}}+ \|Q\|_{2,\Omega ^{\epsilon}}\|\nabla ^{2}\boldsymbol{\eta}\|_{1,\Omega ^{\epsilon}}^{2}\|\mathbf{v}\|_{3,\Omega ^{\epsilon}}+\|\boldsymbol{\eta}\|_{3,\Omega ^{\epsilon}} ^{2}\|Q\|_{2,\Omega ^{\epsilon}}\|\mathbf{v}\|_{2.75,\Omega ^{\epsilon}})\mathrm{d}t  
     \nonumber \\ 
     & \displaystyle\leq  C 
\delta   \int_0^T \| Q\|^2_{2, \Omega ^ \epsilon } \mathrm{d} t + C_ \delta  T \mathcal{P}( \sup_{t \in [0,T]} \mathcal{E}^ \epsilon (t)),
\end{align}
where we have used \eqref{bd-te-l}, the Sobolev embedding theorem and the fact 
\begin{align}\label{eta-tim-esti}
\displaystyle \|\nabla ^{2}\boldsymbol{\eta}\|_{1,\Omega ^{\epsilon}} \leq C\int _{0}^{T}\|\mathbf{v}\|_{3}\mathrm{d}t \leq C T ^{\frac{1}{2}}\left( \int _{0}^{T}\| \mathbf{v}\|_{3}^{2}\mathrm{d}t  \right)^{\frac{1}{2}} 
\end{align}
due to $ \boldsymbol{\eta}(\mathbf{x},0)=\mathbf{x} $ for any $ \mathbf{x} \in \Omega ^{\epsilon} $. Inserting the estimates of $ \mathcal{R}_{1}^{a} $ and $ \mathcal{R}_{1}^{b} $ into \eqref{cal-R-1}, we thus have 
\begin{align}
\displaystyle  \mathcal{R} _1^a \leq C_ \delta  T P( \sup_{t \in [0,T]} \mathcal{E}^ \epsilon (t)) + \delta   \sup_{t \in [0,T]} \mathcal{E}^ \epsilon (t)
\end{align}
Next, for the integral $ \mathcal{R} _2$, it holds that 
\begin{align}\label{CAL-R-2}
\mathcal{R} _2 & \le  \underbrace{\int_0^T  \int_{B^+} \left \vert    (A^k_{\lambda} A^j_{\lambda} - \delta^{kj}) \bar \p^2 v^i,_j\, [ \zeta ^2 \bar \partial^2v^i],_k \right \vert \mathrm{d}\mathbf{y} \mathrm{d}t}_{ \mathcal{R} _2^a}
+  \underbrace{\int_0^T  \int_{B^+} \left \vert  \bar \p^2  (A^k_{\lambda} A^j_{\lambda} - \delta^{kj})  v^i,_j\, [ \zeta ^2 \bar \partial^2v^i],_k \right \vert \mathrm{d}\mathbf{y} \mathrm{d}t }_{ \mathcal{R} _2^b} 
 \nonumber \\ 
 & \qquad + \underbrace{2 \int_0^T  \int_{B^+} \left \vert  \bar \p  (A^k_{\lambda} A^j_{\lambda} - \delta^{kj}) \bar \p v^i,_j\, [ \zeta ^2 \bar \partial^2v^i],_k \right \vert \mathrm{d}\mathbf{y} \mathrm{d}t}_{ \mathcal{R} _2^c}.
\end{align}
We now estimate the terms on the right hand side of \eqref{CAL-R-2}. Using (\ref{est-A}) and choosing $\vartheta < \delta $, we get 
\begin{align}
\displaystyle  \mathcal{R} _2^a &\leq C\int _{0}^{T}\|A A ^{T}- \mathbb{I}_{3}\|_{L ^{\infty}(\Omega ^{\epsilon})} \left( \|\mathbf{v}\|_{3, \Omega ^{\epsilon}}^{2}+ \|\mathbf{v}\|_{3, \Omega ^{\epsilon}}\|\mathbf{v}\|_{2, \Omega ^{\epsilon}} \right) \mathrm{d}\tau 
 \nonumber \\ 
 & \displaystyle \leq C_ \delta  T P( \sup_{t \in [0,T]} \mathcal{E}^ \epsilon (t)) + \delta   \sup_{t \in [0,T]} \mathcal{E}^ \epsilon (t).
\end{align}
For $ \mathcal{R}_{2}^{b} $ and $ \mathcal{R}_{2}^{c} $, by \eqref{bd-te-l}, \eqref{A-express}, \eqref{na-eta-bd} and \eqref{eta-tim-esti}, we get for any $ \delta>0 $,
\begin{align*}
\displaystyle \mathcal{R}_{2}^{b}& \leq C\int _{0}^{T}\int _{B^{+}}\left( \vert \bar{\partial}A\vert ^{2}+\vert A\vert \vert \bar{\partial}^{2}A\vert \right)\vert \nabla \mathbf{v}\vert (\vert \nabla \bar{\partial}^{2}\mathbf{v}\vert+\vert \bar{\partial}^{2}\mathbf{v}\vert)  \mathrm{d}\mathbf{y}\mathrm{d}t 
 \nonumber \\ 
 & \displaystyle \leq C\int _{0}^{T} \int _{\Omega ^{\epsilon}} ( \vert \nabla \boldsymbol{\eta}\vert+ \vert \nabla ^{2}\boldsymbol{\eta}\vert ^{2}+\vert \nabla ^{3}\boldsymbol{\eta}\vert) \vert \nabla \mathbf{v}\vert (\vert \nabla \mathbf{v}\vert + \vert \nabla ^{2}\mathbf{v}\vert+\vert \nabla ^{3}\mathbf{v}\vert)\mathrm{d}\mathbf{x}\mathrm{d}t 
  \nonumber \\ 
     &\leq C\int _{0}^{T}  (\| \nabla ^{2}\boldsymbol{\eta}\|_{1,\Omega ^{\epsilon}} ^{2}+\| \nabla ^{2}\boldsymbol{\eta}\|_{1,\Omega ^{\epsilon}})\|\mathbf{v}\|_{3,\Omega ^{\epsilon}}^{2}\mathrm{d}t \leq C_ \delta  T P( \sup_{t \in [0,T]} \mathcal{E}^ \epsilon (t)) + \delta   \sup_{t \in [0,T]} \mathcal{E}^ \epsilon (t)
\end{align*}
 and  
\begin{align*}
\displaystyle \mathcal{R}_{2}^{c} & \leq C\int _{0}^{T} \int _{B ^{+}}\vert A\vert \vert \bar{\partial}A\vert \vert \nabla \mathbf{v}\vert  (\vert \nabla \bar{\partial}^{2}\mathbf{v}\vert+\vert \bar{\partial}^{2}\mathbf{v}\vert)  \mathrm{d}\mathbf{y}\mathrm{d}t 
 \nonumber \\ 
 & \displaystyle \leq C\int _{0}^{T}(\|\nabla ^{2}\boldsymbol{\eta}\|_{L ^{4}(\Omega ^{\epsilon})}\|\nabla \mathbf{v}\|_{L ^{4}(\Omega ^{\epsilon})}\|\mathbf{v}\|_{3, \Omega ^{\epsilon}}+\|\mathbf{v}\|_{2,\Omega ^{\epsilon}}\|\mathbf{v}\|_{3,\Omega ^{\epsilon}} )\mathrm{d}t 
  \nonumber \\ 
  & \displaystyle \leq  C_ \delta  T P( \sup_{t \in [0,T]} \mathcal{E}^ \epsilon (t)) + \delta   \sup_{t \in [0,T]} \mathcal{E}^ \epsilon (t).
\end{align*}
Therefore we have from \eqref{CAL-R-2} that 
\begin{align}
\displaystyle  \mathcal{R} _2 \leq C_ \delta  T P( \sup_{t \in [0,T]} \mathcal{E}^ \epsilon (t)) + \delta   \sup_{t \in [0,T]} \mathcal{E}^ \epsilon (t)
\end{align}
for any $ \delta>0 $. The integral  $ \mathcal{R} _3$ is straightforward and satisfies
\begin{align}\label{cs5}
\mathcal{R} _3 & \leq C\int_0^T\int _{B ^{+}}  \bar{\partial}^{2}(A^k_{\lambda} A^j_{\lambda} v ^{i}, _{j}) \zeta \zeta, _{k}\bar \partial^2v^i\mathrm{d}\mathbf{y} \mathrm{d}t
 \nonumber \\
 & \displaystyle \leq C \int _{0}^{T} \int _{B^{+}}\vert \bar{\partial}^{2}A\vert \vert \nabla \mathbf{v}\vert \vert \bar{\partial}^{2}\mathbf{v}\vert \mathrm{d}\mathbf{y} \mathrm{d}t +C \int _{0}^{T}\int _{B^{+}}\vert \bar{\partial}A\vert \vert \nabla  ^{2}\mathbf{v}\vert ^{2}\mathrm{d}\mathbf{y} \mathrm{d}t
  \nonumber \\
  & \displaystyle \quad+ C \int _{0}^{T}\int _{B^{+}}\vert \nabla ^{3}\mathbf{v}\vert \vert \nabla ^{2}\mathbf{v}\vert \mathrm{d}\mathbf{y} \mathrm{d}t
   \nonumber \\
    &\le    C_ \delta  T P( \sup_{t \in [0,T]} \mathcal{E}^ \epsilon (t)) + C\delta   \sup_{t \in [0,T]} \mathcal{E}^ \epsilon (t) 
\end{align}
for any $ \delta>0 $. For $ \mathcal{R}^{4} $, we utilize \eqref{bd-te-l}, \eqref{localMHD-B}, \eqref{localMHD-B-bdy},  \eqref{na-eta-bd}, \eqref{A-express} and the Sobolev embedding theorem to derive that  
\begin{align}
\displaystyle \mathcal{R}_{4}&= \frac{1}{4 \pi} \int_0^t\int _{B ^{+}}\zeta ^{2}\bar{\partial}^{2} [H^{i}(A ) _{i}^{\lambda}H^{j} ,_{\lambda}] \bar \partial^2 v ^{j} \mathrm{d}\mathbf{y}\mathrm{d}s
 \nonumber \\
 & \displaystyle =- \frac{1}{4 \pi} \int _{0}^{t} \int _{B ^{+}}\bar{\partial}^{2}[H^{i}A _{i}^{\lambda}H ^{j}](2 \zeta \zeta, _{\lambda}\bar{\partial}^{2}v ^{j}+\zeta ^{2}\bar{\partial}^{2}v ^{j}, _{\lambda})\mathrm{d}\mathbf{y} \mathrm{d}s
  \nonumber \\
  & \displaystyle \quad  \leq- \frac{1}{4 \pi} \int _{0}^{t}\int _{B ^{+}}\zeta ^{2} H ^{i}A _{i}^{\lambda}\bar{\partial}^{2}H  ^{j} \bar{\partial}^{2}v ^{j}, _{\lambda}\mathrm{d}\mathbf{y} \mathrm{d}s+C \int _{0}^{T} \int _{B ^{+}} \vert \bar{\partial}\mathbf{H}\vert ^{2}\left( \vert \bar{\partial}^{2} \mathbf{v}\vert+\vert \nabla ^{3}\mathbf{v}\vert \right)\mathrm{d}\mathbf{y} \mathrm{d}t
   \nonumber \\
   & \displaystyle \quad \quad +\int _{0}^{T} \int _{B ^{+}} \vert \bar{\partial}A\vert \vert \mathbf{H}\vert \vert \bar{\partial}\mathbf{H}\vert\left( \vert \bar{\partial}^{2} \mathbf{v}\vert+\vert \nabla ^{3}\mathbf{v}\vert \right)\mathrm{d}\mathbf{y} \mathrm{d}t
    \nonumber \\
    & \displaystyle \leq  - \frac{1}{4 \pi} \int _{0}^{t}\int _{B ^{+}}\zeta ^{2} H ^{i}A _{i}^{\lambda}\bar{\partial}^{2}H  ^{j} \bar{\partial}^{2}v ^{j}, _{\lambda}\mathrm{d}\mathbf{y} \mathrm{d}s +
    \int _{0}^{T} \left( \|\nabla \mathbf{H}\|_{0,\Omega ^{\epsilon}}+ \|\boldsymbol{\eta}\|_{2, \Omega ^{\epsilon}}\|\mathbf{H}\|_{2,\Omega ^{\epsilon}}^{2} \right)\|\mathbf{v}\|_{3,\Omega ^{\epsilon}}\mathrm{d}t 
     \nonumber \\ 
     & \displaystyle \leq   - \frac{1}{4 \pi} \int _{0}^{t}\int _{B ^{+}}\zeta ^{2} H ^{i}A _{i}^{\lambda}\bar{\partial}^{2}H  ^{j} \bar{\partial}^{2}v ^{j}, _{\lambda}\mathrm{d}\mathbf{y} \mathrm{d}s+\mathcal{M}_0 +  C_ \delta  T P( \sup_{t \in [0,T]} \mathcal{E}^ \epsilon (t)) + C\delta   \sup_{t \in [0,T]} \mathcal{E}^ \epsilon (t) 
\end{align}
for any $ \delta>0 $. Similarly, we have for $ \mathcal{R}_{5} $ that
\begin{align*}
\displaystyle  \mathcal{R}_{5}&= \frac{1}{4 \pi}\int_0^t\int _{B ^{+}}\zeta ^{2}\bar{\partial}^{2}[ H ^{k}A _{k}^{\lambda}v  ^{i}, _{\lambda}] \bar{\partial}^{2} H ^{i}\mathrm{d}\mathbf{y} \mathrm{d}s
 \nonumber \\
 & \displaystyle \leq  \frac{1}{4 \pi}\int_0^t\int _{B ^{+}}\zeta ^{2} H ^{k}A _{k}^{\lambda}\bar{\partial}^{2}v ^{i}, _{\lambda} \bar{\partial}^{2} H ^{i}\mathrm{d}\mathbf{y} \mathrm{d}s +C  \int _{0}^{T}\int _{B ^{+}}\vert \nabla ^{2} \mathbf{H}\vert ^{2}\vert \nabla \mathbf{v}\vert \mathrm{d}\mathbf{y} \mathrm{d}s
  \nonumber \\
  & \displaystyle \quad +C \int _{0}^{T} \int _{B ^{+}}\left( \vert \bar{\partial}^{2}A\vert  \vert \mathbf{H}\vert \vert \nabla \mathbf{v}\vert+\vert \nabla \mathbf{H}\vert \vert \nabla ^{2}\mathbf{v}\vert+\vert \nabla \mathbf{H}\vert \vert \bar{\partial}A\vert \vert \nabla \mathbf{v}\vert\right)\vert \nabla ^{2}\mathbf{H}\vert \mathrm{d}\mathbf{y} \mathrm{d}t
   \nonumber \\
   & \displaystyle \leq \frac{1}{4 \pi}\int_0^T\int _{B ^{+}}\zeta ^{2} H ^{k}A  _{k}^{\lambda}\bar{\partial}^{2}v  ^{i}, _{\lambda} \bar{\partial}^{2} H ^{i}\mathrm{d}\mathbf{y} \mathrm{d}t+ C
   \int _{0}^{T} \|\mathbf{H}\|_{2,\Omega ^{\epsilon}}^{2}\|\mathbf{v}\|_{3,\Omega ^{\epsilon}} \mathrm{d}t 
    \nonumber \\ 
    & \displaystyle \quad+C \int _{0}^{T}\left( \|\boldsymbol{\eta}\| _{3,\Omega ^{\epsilon}}+\|\boldsymbol{\eta}\|_{3,\Omega ^{\epsilon}}^{2}+1 \right) \|\mathbf{H}\|_{2,\Omega ^{\epsilon}}^{2}\|\mathbf{v}\|_{3,\Omega ^{\epsilon}} \mathrm{d}t 
    + C \int _{0}^{T}\|\mathbf{H}\|_{2,\Omega ^{\epsilon}} ^{2} \|\boldsymbol{\eta}\| _{3,\Omega ^{\epsilon}} \|\mathbf{v}\| _{3,\Omega ^{\epsilon}}\mathrm{d}t   \nonumber \\ 
    & \displaystyle \leq \frac{1}{4 \pi}\int_0^T\int _{B ^{+}}\zeta ^{2} H ^{k}A  _{k}^{\lambda}\bar{\partial}^{2}v  ^{i}, _{\lambda} \bar{\partial}^{2} H^{i}\mathrm{d}\mathbf{y} \mathrm{d}t + \mathcal{M}_0 +  C_ \delta  T P( \sup_{t \in [0,T]} \mathcal{E}^ \epsilon (t)) + C\delta   \sup_{t \in [0,T]} \mathcal{E}^ \epsilon (t).
\end{align*}
Summing over all of the  boundary charts $l=1,...,K$ in (\ref{cs7}), using the estimates for $ \mathcal{R} _{i}\,(i=1,\cdots,5) $ together with the trace theorem, Lemma \ref{lem-trace}, we get for any $ \delta>0 $,
\begin{equation*}
 \int_0^T   \|   \mathbf{v}(\cdot , t)  \|^2_{2.5, \Gamma^ \epsilon } \mathrm{d}t\le
\mathcal{M}_0 +  C_ \delta  T P( \sup_{t \in [0,T]} \mathcal{E}^ \epsilon (t)) + \delta   \sup_{t \in [0,T]} \mathcal{E}^ \epsilon (t).
\end{equation*}
This ends the proof of Lemma \ref{lem-bdy-regul}.
\end{proof}
Next we shall establish some estimates for the magnetic field $ \mathbf{H} $.

\begin{lemma}\label{lem-H}
There exists a polynomial function $P$ and a $C_\delta >0$ independent of $\epsilon$ such that
\begin{equation}
\sup_{t \in [0,T]}   \| \mathbf{H} ( \cdot ,t) \|^2_{2, \Omega^\epsilon } \le \mathcal{M}_0 +  C_ \delta  T P( \sup_{t \in [0,T]} \mathcal{E}^ \epsilon (t)) + \delta   \sup_{t \in [0,T]} \mathcal{E}^ \epsilon (t)
\end{equation}
for any positive number $\delta$.
\end{lemma}

\begin{proof}
We will show only the second-order estimate since the lower-order estimates can be derived similarly. Applying $ \partial ^{2} $ to the equation \eqref{mhd-lagran-c}, and testing the resulting equation against $ \partial ^{2}\mathbf{H} $, we get 
\begin{equation}\label{h1}
\int_{ \Omega^\epsilon }\p^2 \mathbf{H} _{t}  \cdot \p^2 \mathbf{H} \mathrm{d}\mathbf{x}=\int_{ \Omega^\epsilon } \p^2 (H ^{i}A _{i}^{k}\mathbf{v}, _{k}) \cdot \p^2 \mathbf{H} \mathrm{d}\mathbf{x}.
\end{equation}
Integrating \eqref{h1} in time, we get
\begin{align}\label{pa-2-H-iden}
&\frac{1}{2}  \| \nabla ^{2} \mathbf{H} (t) \|^2_{0, \Omega^\epsilon }=  \frac{1}{2} \| \nabla^2 \mathbf{H}_0 \|^2_{0, \Omega^\epsilon } + \int_0^t  \int _{\Omega ^{\epsilon}}  \partial^2 (H ^{i}A _{i}^{k}\mathbf{v}, _{k}) \cdot \partial^2 \mathbf{H} \mathrm{d}\mathbf{x}\mathrm{d}s\\
 & =\frac{1}{2} \| \nabla^2 \mathbf{H}_0 \|^2_{0, \Omega^\epsilon } +  \underbrace{\int_0^t  \int _{\Omega ^{\epsilon}}  \partial^2 H ^{i}A _{i}^{k}\mathbf{v}, _{k} \cdot \partial^2 \mathbf{H} \mathrm{d}\mathbf{x}\mathrm{d}s }_{ \mathcal{U} _1} 
 +  \underbrace{\int_0^t  \int _{\Omega ^{\epsilon}}    H ^{i} \partial^2 A _{i}^{k}\mathbf{v}, _{k} \cdot \partial^2 \mathbf{H} \mathrm{d}\mathbf{x}\mathrm{d}s}_{ \mathcal{U} _2} 
  \nonumber \\ 
  & \displaystyle \quad+  \underbrace{\int_0^t  \int _{\Omega ^{\epsilon}}   H ^{i}A _{i}^{k} \partial^2 \mathbf{v}, _{k} \cdot \partial^2 \mathbf{H} \mathrm{d}\mathbf{x}\mathrm{d}s}_{ \mathcal{U} _3}+ 2 \underbrace{\int_0^t  \int _{\Omega ^{\epsilon}} \partial H ^{i} \partial A _{i}^{k}\mathbf{v}, _{k} \cdot\partial^2 \mathbf{H} \mathrm{d}\mathbf{x}\mathrm{d}s}_{ \mathcal{U} _4} 
  \nonumber \\ 
  & \displaystyle \quad
   +2  \underbrace{\int_0^t  \int _{\Omega ^{\epsilon}}  \partial H ^{i}A _{i}^{k} \partial \mathbf{v}, _{k} \cdot \partial^2 \mathbf{H}\mathrm{d}\mathbf{x}\mathrm{d}s} _{ \mathcal{U} _5}+ 2  \underbrace{\int_0^t  \int _{\Omega ^{\epsilon}}   H ^{i} \partial A _{i}^{k} \partial \mathbf{v}, _{k} \cdot \partial^2 \mathbf{H} \mathrm{d}\mathbf{x}\mathrm{d}s}_{ \mathcal{U} _6}.
\end{align}
We now estimate $ \mathcal{U}_{i}\,(i=1,\cdots,6) $ term by term. By \eqref{bd-te-l}, \eqref{na-eta-bd}, \eqref{A-express} and the Sobolev embedding theorem, we get 
\begin{align}
\displaystyle \mathcal{U}_{1} &\leq C\int _{0}^{t}\|\nabla^{2}\mathbf{H}\|_{0,\Omega ^{\epsilon}}^{2}\|\nabla \mathbf{v}\|_{L ^{\infty}(\Omega ^{\epsilon})}\mathrm{d}s  \leq \delta \int _{0}^{t}\|\mathbf{v}\|_{3, \Omega ^{\epsilon}}^{2}\mathrm{d}s +C _{\delta} \int _{0}^{t}\| \nabla ^{2} \mathbf{H}\|_{0,\Omega ^{\epsilon}}^{4}\mathrm{d}s 
 \nonumber \\ 
 & \displaystyle  \leq C_ \delta  T P( \sup_{t \in [0,T]} \mathcal{E}^ \epsilon (t)) + \delta   \sup_{t \in [0,T]} \mathcal{E}^ \epsilon (t).
\end{align}
Similarly, we have 
\begin{align*}
\displaystyle \mathcal{U}_{3}& \leq C\int _{0}^{t} \|\mathbf{H}\|_{L ^{\infty}(\Omega ^{\epsilon})} \|\mathbf{v}\|_{0,\Omega^\epsilon} \|\mathbf{H}\|_{2, \Omega^\epsilon} \mathrm{d}s \leq C \int _{0}^{t}\|\mathbf{H}\|_{2,\Omega ^{\epsilon}}^{2}\|\mathbf{v}\|_{3, \Omega ^{\epsilon}}\mathrm{d}s \nonumber \\ 
 & \leq C_ \delta  T P( \sup_{t \in [0,T]} \mathcal{E}^ \epsilon (t)) + \delta   \sup_{t \in [0,T]} \mathcal{E}^ \epsilon (t)
\end{align*}
and  
\begin{align*}
\displaystyle \mathcal{U}_{5}& \leq C\int _{0}^{t} \| \nabla \mathbf{H}\|_{L ^{4}(\Omega ^{\epsilon})}\|\nabla^{2}\mathbf{v}\|_{L ^{4}(\Omega ^{\epsilon})} \| \nabla^{2}\mathbf{H}\|_{0,\Omega ^{\epsilon}}\mathrm{d}s \leq C \int _{0}^{t}\|\mathbf{H}\|_{2,\Omega ^{\epsilon}}^{2}\|\mathbf{v}\|_{3, \Omega ^{\epsilon}}\mathrm{d}s \nonumber \\ 
 & \leq C_ \delta  T P( \sup_{t \in [0,T]} \mathcal{E}^ \epsilon (t)) + \delta   \sup_{t \in [0,T]} \mathcal{E}^ \epsilon (t).
\end{align*}
For $ \mathcal{U}_{2} $, from \eqref{na-eta-bd} and \eqref{A-express}, we get 
\begin{align*}
\displaystyle \vert \nabla ^{2}A\vert \leq C\vert \nabla ^{3}\boldsymbol{\eta}\vert +C \vert \nabla ^{2}\boldsymbol{\eta}\vert ^{2}.
\end{align*}
This along with \eqref{bd-te-l}, the Cauchy-Schwarz inequality and the Sobolev embedding theorem implies that 
\begin{align}\label{esti-CAL-U2}
 \displaystyle \mathcal{U}_{2}& \leq C\int _{0}^{t} \int _{\Omega ^{\epsilon}}\vert \mathbf{H}\vert  \left(\vert \nabla ^{3}\boldsymbol{\eta}\vert + \vert \nabla ^{2}\boldsymbol{\eta}\vert ^{2}  \right) \vert \nabla \mathbf{v}\vert \vert \nabla ^{2}\mathbf{H}\vert \mathrm{d}\mathbf{x}\mathrm{d}s 
  \nonumber \\ 
  & \displaystyle \leq C\int _{0}^{t}\|\mathbf{H}\|_{L ^{\infty}(\Omega ^{\epsilon})} (\|\boldsymbol{\eta}\|_{3, \Omega ^{\epsilon}}+\|\boldsymbol{\eta}\|_{3, \Omega ^{\epsilon}}^{2}) \| \nabla \mathbf{v}\|_{L ^{\infty}(\Omega ^{\epsilon})}\|\partial ^{2}\mathbf{H}\|_{0, \Omega ^{\epsilon} }\mathrm{d}s 
   \nonumber \\ 
   & \displaystyle \leq C\int _{0}^{t}\|\mathbf{H}\|_{2, \Omega ^{\epsilon}}^{2}(\|\boldsymbol{\eta}\|_{3, \Omega ^{\epsilon}}+\|\boldsymbol{\eta}\|_{3, \Omega ^{\epsilon}}^{2})\|\mathbf{v}\|_{3,\Omega ^{\epsilon}}\mathrm{d}s 
    \nonumber \\ 
    & \displaystyle \leq C_ \delta  T P( \sup_{t \in [0,T]} \mathcal{E}^ \epsilon (t)) + \delta   \sup_{t \in [0,T]} \mathcal{E}^ \epsilon (t).
 \end{align}
 By similar arguments as in \eqref{esti-CAL-U2}, we have for $ \mathcal{U}_{4} $ and $ \mathcal{U}_{6} $ that 
 \begin{align}
 \displaystyle \mathcal{U}_{4}& \leq   C\int _{0}^{t} \int _{\Omega ^{\epsilon}}\vert \nabla\mathbf{H}\vert \vert \partial ^{2}\boldsymbol{\eta}\vert  \vert \nabla \mathbf{v}\vert \vert \nabla ^{2}\mathbf{H}\vert \mathrm{d}\mathbf{x}\mathrm{d}s 
  \nonumber \\ 
  & \displaystyle \leq C\int _{0}^{t} \| \nabla \mathbf{H}\|_{L ^{4}(\Omega ^{\epsilon})}\| \nabla ^{2}\boldsymbol{\eta}\|_{L ^{4}(\Omega ^{\epsilon})} \| \nabla \mathbf{v}\|_{L ^{\infty}(\Omega ^{\epsilon})} \| \nabla ^{2}\mathbf{H}\|_{0,\Omega ^{\epsilon}} \mathrm{d}s 
   \nonumber \\ 
   & \displaystyle \leq C\int _{0}^{t} \|\mathbf{H}\|_{2, \Omega ^{\epsilon}}^{2} \|\boldsymbol{\eta}\|_{3,\Omega ^{\epsilon}} \|\mathbf{v}\| _{3, \Omega ^{\epsilon}} \mathrm{d}s 
    \nonumber \\ 
     & \leq C_ \delta  T P( \sup_{t \in [0,T]} \mathcal{E}^ \epsilon (t)) + \delta   \sup_{t \in [0,T]} \mathcal{E}^ \epsilon (t)
 \end{align}
 and 
 \begin{align}
 \displaystyle  \mathcal{U}_{6}& \leq   C\int _{0}^{t} \int _{\Omega^\epsilon}\vert \mathbf{H}\vert \vert \nabla ^{2}\boldsymbol{\eta}\vert  \vert \nabla ^{2} \mathbf{v}\vert \vert \nabla ^{2}\mathbf{H}\vert \mathrm{d}\mathbf{x}\mathrm{d}s 
  \nonumber \\ 
 \displaystyle & \leq C\int _{0}^{t}\| \mathbf{H}\|_{L ^{\infty}(\Omega ^{\epsilon})}  \|\nabla ^{2} \mathbf{v}\|_{L ^{4}(\Omega ^{\epsilon})}\| \nabla ^{2}\boldsymbol{\eta}\|_{L ^{4}(\Omega ^{\epsilon})} \|\nabla ^{2}\mathbf{H}\|_{0,\Omega} \mathrm{d}s 
   \nonumber \\ 
   & \displaystyle \leq C\int _{0}^{t} \|\mathbf{H}\|_{2,\Omega ^{\epsilon}}^{2}\|\mathbf{v}\|_{3, \Omega ^{\epsilon}}\|\boldsymbol{\eta}\|_{3, \Omega ^{\epsilon}}\mathrm{d}s 
    \nonumber \\ 
    & \leq C_ \delta  T P( \sup_{t \in [0,T]} \mathcal{E}^ \epsilon (t)) + \delta   \sup_{t \in [0,T]} \mathcal{E}^ \epsilon (t).
 \end{align}
 Inserting the estimates of $ \mathcal{U}_{i}\,(1,\cdots,6) $ into \eqref{pa-2-H-iden}, we then get for any $ t \in [0,T] $,
 \begin{align}
 \displaystyle  \| \nabla ^{2}\mathbf{H}(\cdot,t)\| _{0,\Omega ^{\epsilon}} \leq \mathcal{M}_{0} +C_ \delta  T P( \sup_{t \in [0,T]} \mathcal{E}^ \epsilon (t)) + \delta   \sup_{t \in [0,T]} \mathcal{E}^ \epsilon (t).
 \end{align}
Following the above procedure, we can derive the estimates for the lower-order derivatives of the magnetic field, and ultimately complete the proof of the present lemma.
\end{proof}

In the following, we shall derive the estimates for the time-derivative of $ (\mathbf{u},\mathbf{H}) $. To do so, we differentiate \eqref{localNS} with respect to the time variable, and obtain the following equations:
\begin{subequations}
\label{NSlagt}
\begin{alignat}{2}
\partial _{t}\boldsymbol{\eta}  & = \mathbf{v}  \ && \text{ in } \Omega^ \epsilon  \times [0,T],  \label{eta-eq}\\
 \partial _{t}^{2} \mathbf{v}   - \Delta _ {\boldsymbol{\eta}} \partial _{t}\mathbf{v} + A^T \nabla Q_t   &= -  A^T _{t} \nabla Q  + [ \partial_t( A^j_s A^k_s) \mathbf{v},_k],_j+\frac{1}{\mu _{0}}\partial _{t}\left[  H ^{i}A _{i}^{l}\mathbf{H}, _{l} \right]  \ \ && \text{ in } \Omega^ \epsilon  \times (0,T] ,  \\
\operatorname{div} _{\boldsymbol{\eta}} \mathbf{v}_t &=-v^i,_j \partial_t A^j_i \ \ && \text{ in } \Omega^\epsilon  \times [0,T] , \label{tim-v-eq}  \\
\displaystyle
\partial _{t}^{2} \mathbf{H} &=\partial _{t}H ^{k}A  _{k}^{\lambda}\mathbf{v} , _{\lambda}+ H ^{k}\partial _{t}(A  _{k}^{\lambda}\mathbf{v}  , _{\lambda})
, \label{H-t-EQ}\ \ && \text{ in } \Omega^\epsilon  \times [0,T] ,  \\
\displaystyle \operatorname{div} _{\boldsymbol{\eta}} \mathbf{H}_t &=-H^i,_j \partial_t A^j_i \ \ && \text{ in } \Omega^\epsilon  \times [0,T] ,  \\
\partial_t\left[ \operatorname{Def} _ {\boldsymbol{\eta}} \mathbf{v} \cdot \mathbf{n} - Q \mathbf{n}\right] &=0\ \ && \text{ on } \Gamma^ \epsilon  \times [0,T] \,,  \\
(\boldsymbol{\eta},\mathbf{v}, \partial _{t}\mathbf{v},\mathbf{H} , \partial _{t}\mathbf{H} )  &=(\mathbf{e},\mathbf{u}_0^ \epsilon ,\mathbf{u}_1^ \epsilon, \mathbf{H}_{0},\mathbf{H} _{1} ) \ \  \ \ && \text{ in } \Omega^\epsilon  \times \{t=0\} , 
 \nonumber 
\end{alignat}
\end{subequations}
Here we define $u_1^ \epsilon = \nu\Delta \mathbf{u}_0^ \epsilon - \nabla P_0^ \epsilon + \frac{1}{4 \pi}\mathbf{H}_{0}^{\epsilon} \cdot \nabla \mathbf{H}_{0}^{\epsilon}$, and $ \mathbf{H}_{1}^{\epsilon}=\mathbf{H}_{0}^{\epsilon}\cdot \nabla \mathbf{u}_{0}^{\epsilon}$. So we have
\begin{equation}\label{cs10}
\| u_1^ \epsilon \|_{0,\Omega^\epsilon} +  \|  \mathbf{H} _{1}^ \epsilon\|_{0,\Omega^\epsilon} \le C \mathcal {P}(\mathcal{E}(0)),
\end{equation}
where the coefficients of $\mathcal {P}$ are independent of $\epsilon$. The estimates on $ \mathbf{H}_{t} $ and $ \mathbf{v}_{t} $ can be stated in the following lemma.
\begin{lemma}\label{lem-tim-deriv}
The following estimates holds:
\begin{align}\label{tim-coclu-lem}
&\sup_{t \in [0,T]}  ( \| \mathbf{v}_t ( \cdot ,t) \|^2_{0, \Omega^\epsilon }+\|\mathbf{H} _{t}(\cdot,t)\| _{0,\Omega ^{\epsilon}}^{2})+ \int_0^T\| \mathbf{v}_t\|^2_{1, \Omega^\epsilon } \mathrm{d}t 
 \nonumber \\ 
 &~\displaystyle
 \leq 
\mathcal{M}_0 +    C _{\delta} T ^{1/2} \mathcal{P} (\sup_{t \in [0,T]}  \mathcal{E}^ \epsilon (t)) +
C\delta \sup_{t \in [0,T]} \mathcal{E}^ \epsilon (t) .
\end{align}
\end{lemma}
\begin{proof}
  Define the space of $ \mathop{\mathrm{div}}\nolimits _{\boldsymbol{\eta}} $ free vector fields on $ \Omega ^{\epsilon} $ as 
  \begin{align*}
  \displaystyle \mathcal{V}(t)=\left\{ \boldsymbol{\phi}=(\phi ^{1}(\cdot,t),\phi ^{2}(\cdot,t),\phi ^{3}(\cdot,t)) \in H ^{1}(\Omega ^{\epsilon};\mathbb{R}^{3})\vert\, \mathop{\mathrm{div}}\nolimits _{\boldsymbol{\eta}}\boldsymbol{\phi} =0 \right\}. 
  \end{align*}
  Testing the equation \eqref{tim-v-eq} against any $ \boldsymbol{\phi} \in \mathcal{V} $, we get 
  \begin{align}\label{test-tim-eq}
  \int_{ \Omega^\epsilon } \partial _{t}^{2}\mathbf{v} \cdot \boldsymbol{\phi} \mathrm{d} \mathbf{x} + \int_{\Omega ^ \epsilon } \p_t [A^k_{\lambda} A^j_{\lambda} v^i,_j] \, \phi^i,_k \mathrm{d} \mathbf{x} = \int_ {\Omega^\epsilon} Q\, \partial_t A^k_i \phi^i,_k \, \mathrm{d} \mathbf{x} +\frac{1}{4 \pi}\int _{\Omega ^{\epsilon}}\left[  H ^{i}A _{i}^{\lambda}\mathbf{H}, _{\lambda} \right]_{t}  \cdot \boldsymbol{\phi}\mathrm{d}\mathbf{x} .
  \end{align}
  As in \cite{Shkoller-2019-APAN}, define a vector field $\mathbf{w}$ through
\begin{subequations}
  \label{w-problem}
\begin{alignat}{2}
\operatorname{div} _ {\boldsymbol{\eta}}  \mathbf{w}  &= -v^i,_j \partial_t A^j_i    \ \  \ &&\text{in} \ \ \Omega^ \epsilon  ,\\
\mathbf{w}  &=  \phi(t) \mathbf{n}  \ \ &&\text{on} \ \ \Gamma^ \epsilon   ,
\end{alignat}
\end{subequations}
where $\phi(t) = -\int_{ \Omega^\epsilon } -v^i,_j \partial_t A^j_i \mathrm{d}\mathbf{x} / | \Gamma^ \epsilon |$. A solution $ \mathbf{w} $ of \eqref{w-problem} can be found by solving a Stokes-type problem, and $ \mathbf{w} $ satisfies the following estimates (cf. \cite{Shkoller-2019-APAN} or Lemma 3.2 in \cite{Shkoller-2010-Sima}): 
\begin{gather}\label{est-w}
\| \mathbf{w}( \cdot , t) \|_{k, \Omega^\epsilon } \le C\left( \|  v^i,_j (\cdot , t) \,  \p_t A^j_i(\cdot , t) \|_{k-1, \Omega^\epsilon } +  \| \phi(t) \mathbf{n} \|_{k-1/2, \Gamma^\epsilon }\right) 
\end{gather}
for $ k \geq 1 $, where $ C $ is a positive constant independent of $ \epsilon $. This along with \eqref{A-express} and \eqref{eta-eq} further implies that 
\begin{align}\label{w-esti-small}
\displaystyle \sup _{t \in [0,T]} \|\mathbf{w}(\cdot,t)\| _{1,\Omega ^{\epsilon}}^{2}+ \int _{0}^{T} \|\mathbf{w}(\cdot,t)\| _{2,\Omega ^\epsilon}^{2} \mathrm{d}t \leq \mathcal{M}_{0}+T ^{1/2}\mathcal{P}( \sup_{t \in [0,T]} \mathcal{E}^ \epsilon (t)).
\end{align}
Similarly, we know that $ \mathbf{w}_{t} $ solves 
\begin{subequations}
  \label{wt}
\begin{alignat}{2}
\mathop{\mathrm{div}}\nolimits _{\boldsymbol{\eta}} \mathbf{w}_t  &=  -\left(w^i,_j \p_t A^j_i + \p_t(v^i,_j \p_t A^j_i )\right)   \ \  \ &&\text{in} \ \ \Omega^ \epsilon,\\
\mathbf{w}_t  &=  \left( \phi n \right) _t \ \ &&\text{on} \ \ \Gamma^ \epsilon,
\end{alignat}
\end{subequations}
and satisfies the estimates 
\begin{align*}
\displaystyle \|\mathbf{w}_{t}\|_{1,\Omega ^{\epsilon}}  \leq C \left( \|w^i,_j \p_t A^j_i + \p_t(v^i,_j \p_t A^j_i)\| _{0,\Omega ^{\epsilon}}+\|(\phi \mathbf{n})_{t}\| _{1/2, \Gamma ^{\epsilon}} \right) 
\end{align*}
and 
\begin{align}
\displaystyle \int _{0}^{T}\|\mathbf{w} _{t}\| _{1, \Omega ^{\epsilon}}^{2}\mathrm{d}t \leq \mathcal{P}( \sup_{t \in [0,T]} \mathcal{E}^ \epsilon (t)). 
\end{align}
Notice that $ \mathbf{v}_{t}- \mathbf{w} \in \mathcal{V}(t) $. Then we take $ \boldsymbol{\phi}= \mathbf{v}_{t}- \mathbf{w}$ in \eqref{test-tim-eq} to get 
\begin{align}\label{v-t-iden}
&\displaystyle {\frac{1}{2}} \frac{\mathrm{d}}{\mathrm{d}t} \| \mathbf{v}_t ( \cdot ,t) \|^2_{0, \Omega^\epsilon }+ \int_{\Omega ^ \epsilon } \p_t [A^k_{\lambda} A^j_{\lambda} v^i,_j]\, v_t^i,_k \mathrm{d}\mathbf{x}
\nonumber \\
 &~ \displaystyle  = \int_{ \Omega^\epsilon } \partial _{t}^{2} \mathbf{v} \cdot \mathbf{w} \mathrm{d} \mathbf{x} + \int_{\Omega^\epsilon }  \p_t (A^k_{\lambda} A^j_{\lambda} v^i,_j)  \, w^i,_k \mathrm{d}\mathbf{x} +  \int_ {\Omega^\epsilon} Q\, \p_t A^k_i  \left[v_t^i,_k - w^i,_k \right] \, \mathrm{d}\mathbf{x} \\ 
 & ~ \displaystyle \quad+  \int _{\Omega ^{\epsilon}} \partial _{t} H ^{i}A _{i}^{\lambda}\mathbf{H}, _{\lambda}^{j} (v _{t}^{j}-w^{j})\mathrm{d}\mathbf{x} +\int _{\Omega ^{\epsilon}}  H ^{i}\partial _{t}A _{i}^{\lambda}H, _{\lambda}^{j} (v _{t}^{j}-w^{j})\mathrm{d}\mathbf{x} 
  \nonumber \\ 
  &~\displaystyle \quad+ \int _{\Omega ^{\epsilon}}  H ^{i}A _{i}^{\lambda}\partial _{t}H , _{\lambda}^{j} (v _{t}^{j}-w^{j})\mathrm{d}\mathbf{x} .
\end{align}
Furthermore, testing \eqref{H-t-EQ} against $ \mathbf{H} _{t} $, we get
\begin{align*}
\displaystyle \frac{1}{2}\frac{\mathrm{d}}{\mathrm{d}t}\| \partial_{t}\mathbf{H} \|_{0,\Omega ^{\epsilon}}^{2}=  \int _{\Omega ^{\epsilon}}  \partial _{t}H ^{k}A  _{k}^{\lambda}\mathbf{v}, _{\lambda} \cdot \partial _{t} \mathbf{H} \mathrm{d}\mathbf{x}+ \int _{\Omega ^{\epsilon}}  H ^{k} \partial _{t}A  _{k}^{\lambda}\mathbf{v}  , _{\lambda}\cdot \partial _{t} \mathbf{H} \mathrm{d}\mathbf{x} + \int _{\Omega ^{\epsilon}} H ^{k} A  _{k}^{\lambda}\partial _{t} \mathbf{v}  , _{\lambda}\cdot \partial _{t} \mathbf{H} \mathrm{d}\mathbf{x}.
\end{align*}
This along with \eqref{v-t-iden} implies that 
\begin{align}\label{tim-sum-iden}
& {\frac{1}{2}}  \| \partial _{t}\mathbf{v} ( \cdot ,t) \|^2_{0, \Omega^\epsilon }+ \frac{1}{2}\| \partial_{t} \mathbf{H} \|_{0,\Omega ^{\epsilon}}^{2}+ \int_0^t\| \nabla \partial _{s} \mathbf{v}\|^2_{0, \Omega^\epsilon } \mathrm{d} s  
 \nonumber \\ 
 &~\displaystyle= {\frac{1}{2}}  (\|\mathbf{u}_1 ^{\epsilon} \|^2_{0, \Omega^\epsilon }+ \| \mathbf{H}_{1}^{\epsilon} \|_{0,\Omega ^{\epsilon}}^{2})
 -\underbrace{ \int_0^t \int_{\Omega ^ \epsilon }
 [A^k_{\lambda} A^j_{\lambda}  - \delta ^{kj}] v_{s}^i,_j\, v_{s}^i,_k  \mathrm{d}\mathbf{x} \mathrm{d} s}_{ \mathcal{S} _1}-\underbrace{\int_0^t \int_{\Omega ^ \epsilon } \p_s [A^k_{\lambda} A^j_{\lambda} ] v^i,_j v_s^i,_k \mathrm{d}\mathbf{x} \mathrm{d} s}_{ \mathcal{S} _2} 
  \nonumber \\ 
  & \qquad   
+ \underbrace{\int_0^t \int_{ \Omega^\epsilon } \partial _{s}^{2} \mathbf{v} \cdot \mathbf{w}  \mathrm{d}\mathbf{x} \mathrm{d} s}_{ \mathcal{S} _3}
+ \underbrace{ \int_0^t \int_{\Omega^\epsilon } \p_s [A^k_{\lambda} A^j_{\lambda} v^i,_j]\, w^i,_k  \mathrm{d}\mathbf{x} \mathrm{d} s}_{ \mathcal{S} _4}
+ \underbrace{ \int_0^t \int_ {\Omega^\epsilon} Q\, \p_s A^k_i  \left[ \partial _{s}v^i,_k - w^i,_k \right] \,  \mathrm{d}\mathbf{x} \mathrm{d} s}_{ \mathcal{S} _5} 
 \nonumber \\ 
 & \qquad
 +  \underbrace{\int_0^t  \int _{\Omega ^{\epsilon}}\partial _{s}  H ^{i}A _{i}^{\lambda}\mathbf{H}, _{\lambda} \cdot (\partial _{s}\mathbf{v} - \mathbf{w})\ \mathrm{d}\mathbf{x} \mathrm{d} s}_{ \mathcal{S} _6}
 + \underbrace{\int_0^t  \int _{\Omega ^{\epsilon}} H ^{i}\partial _{s} A _{i}^{\lambda}\mathbf{H}, _{\lambda} \cdot  ( \partial _{s}\mathbf{v} - \mathbf{w})\mathrm{d}\mathbf{x}\mathrm{d}s}_{ \mathcal{S} _7} 
  \nonumber \\ 
  & \qquad- \underbrace{\int_0^t \int _{ \Omega ^{\epsilon}}  H ^{i}A _{i}^{\lambda}\partial _{s} \mathbf{H}, _{\lambda}\cdot  \mathbf{w} \mathrm{d}\mathbf{x} \mathrm{d} s}_{ \mathcal{S} _8}
+\underbrace{ \int_0^t \int _{\Omega ^{\epsilon}}  H ^{i}A _{i}^{\lambda}\partial _{s} \mathbf{H}, _{\lambda} \cdot \mathbf{v} _{s} \mathrm{d}\mathbf{x} \mathrm{d} s}_{ \mathcal{S} _9}+ \underbrace{\int_0^t \int _{\Omega ^{\epsilon}}  \partial _{s}H ^{k}A  _{k}^{\lambda} \mathbf{v} , _{\lambda} \cdot \partial _{s} \mathbf{H} \mathrm{d}\mathbf{x} \mathrm{d} s} _{ \mathcal{S} _{10}} 
 \nonumber \\ 
 & \qquad+\underbrace{ \int_0^t \int _{\Omega ^{\epsilon}}  H ^{k}\partial _{s}A  _{k}^{\lambda}\mathbf{v}  , _{\lambda}\cdot \partial _{s} \mathbf{H}  \mathrm{d}\mathbf{x} \mathrm{d} s} _{ \mathcal{S} _{11}}
+ \underbrace{ \int_0^t \int _{\Omega ^{\epsilon}}  H ^{k}A  _{k}^{\lambda}\partial _{s}\mathbf{v}  , _{\lambda}\cdot \partial _{s} \mathbf{H}  \mathrm{d}\mathbf{x} \mathrm{d} s} _{ \mathcal{S} _{12}}.
\end{align}
We now estimate $ \mathcal{S}_{i}\,(i=1,\cdots,12) $ term by term. Thanks to \eqref{est-A} and the Cauchy-Schwarz inequality, we get for $ \vartheta <1/4 $,
\begin{align}
\displaystyle \mathcal{S}_{1} \leq \vartheta \int _{0}^{t}\|\nabla \mathbf{v}_{s}\|_{0,\Omega ^{\epsilon}}^{2}\mathrm{d}s \leq  \frac{1}{4} \int _{0}^{t}\|\nabla \mathbf{v}_{s}\|_{0,\Omega ^{\epsilon}}^{2}\mathrm{d}s.
\end{align}
Recalling \eqref{A-express}, we know that $\p_t A $ behaves like $ \nabla \eta \, \nabla \mathbf{v}$. Then we derive by virtue of \eqref{bd-te-l}, \eqref{na-eta-bd}, the Cauchy-Schwarz inequality and the Sobolev embedding theorem that
\begin{align}
\displaystyle \vert \mathcal{S}_{2}\vert& \leq C \int _{0}^{T}\int _{\Omega ^{\epsilon}}\vert \nabla  \mathbf{v}\vert ^{2}\vert \nabla \mathbf{v}_{t}\vert \mathrm{d}\mathbf{x}\mathrm{d}t  \leq C\int _{0}^{T} \|\nabla \mathbf{v}\|_{L ^{4}(\Omega ^{\epsilon})}^{2}\|\nabla \mathbf{v}_{t}\|_{0,\Omega ^{\epsilon}} \mathrm{d}t 
 \nonumber \\ 
 & \displaystyle  \leq C_ \delta  T P( \sup_{t \in [0,T]} \mathcal{E}^ \epsilon (t)) + \delta   \sup_{t \in [0,T]} \mathcal{E}^ \epsilon (t).
\end{align}
Similarly, we get 
\begin{align}
\displaystyle  \mathcal{S}_{4} & \leq C\int _{0}^{t}\int _{\Omega  ^{\epsilon}} \left( \vert \nabla \mathbf{v}\vert ^{2}+\vert \nabla \mathbf{v}_{s}\vert \right)  \vert \nabla \mathbf{w}\vert \mathrm{d}\mathbf{x}\mathrm{d}s 
 \nonumber \\ 
 & \displaystyle \leq \frac{1}{4} \int _{0}^{t}\|\nabla \mathbf{v}_{s}\|_{0,\Omega ^{\epsilon}}^{2}\mathrm{d}s+C \int _{0}^{t}\|\nabla \mathbf{w}\|_{0,\Omega ^{\epsilon}}^{2}\mathrm{d}s + C\int _{0}^{t}\|\nabla \mathbf{v}\|_{L ^{4}(\Omega ^{\epsilon})}^{2}\mathrm{d}s \nonumber \\ 
  & \displaystyle  \leq\frac{1}{8} \int _{0}^{t}\|\nabla \mathbf{v}_{s}\|_{0,\Omega ^{\epsilon}}^{2}\mathrm{d}s+\mathcal{M}_{0}+T \mathcal{P}( \sup_{t \in [0,T]} \mathcal{E}^ \epsilon (t)),
  \end{align}
\begin{align}
\displaystyle  \mathcal{S}_{5}& \leq C \int _{0}^{t}\int _{\Omega ^{\epsilon}}\vert Q\vert \vert \nabla \mathbf{v}\vert \left( \vert \nabla \mathbf{v}_{s}\vert+ \vert \nabla \mathbf{w}\vert \right) \mathrm{d}\mathbf{x}\mathrm{d}s 
 \nonumber \\ 
 & \leq  \frac{1}{4} \int _{0}^{t}\|\nabla \mathbf{v}_{s}\|_{0,\Omega ^{\epsilon}}^{2}\mathrm{d}s + C\int _{0}^{t} \|\nabla \mathbf{w}\|_{0,\Omega ^{\epsilon}}^{2}\mathrm{d}s + C \int _{0}^{t} \int _{\Omega ^{\epsilon}}\vert Q\vert ^{2}\vert \nabla \mathbf{v}\vert ^{2}\mathrm{d}\mathbf{x}\mathrm{d}s 
  \nonumber \\ 
  & \displaystyle \leq \frac{1}{4} \int _{0}^{t}\|\nabla \mathbf{v}_{s}\|_{0,\Omega ^{\epsilon}}^{2}\mathrm{d}s + C\int _{0}^{t} \|\nabla \mathbf{w}\|_{0,\Omega ^{\epsilon}}^{2}\mathrm{d}s+ C   \int _{0}^{t}\|Q\| _{L ^{4}(\Omega ^{\epsilon})}^{2}\|\nabla \mathbf{v}\|_{L ^{4}(\Omega ^{\epsilon})}^{2} \mathrm{d}s 
   \nonumber \\ 
   & \displaystyle \leq\frac{1}{4} \int _{0}^{t}\|\nabla \mathbf{v}_{s}\|_{0,\Omega ^{\epsilon}}^{2}\mathrm{d}s+ \mathcal{M}_{0}+T ^{1/2}\mathcal{P}( \sup_{t \in [0,T]} \mathcal{E}^ \epsilon (t))+  C\sup_{t \in [0,T]} \mathcal{E}^ \epsilon (t)\int _{0}^{t}\|Q\|_{1, \Omega ^{\epsilon}}^{2}\mathrm{d}s 
\end{align}
and 
\begin{align}
\displaystyle \mathcal{S}_{7}+\mathcal{S}_{11}& \displaystyle \leq  C \int _{0}^{t} \int _{\Omega ^{\epsilon}} \vert \mathbf{H}\vert \vert \nabla \mathbf{v}\vert (\vert \nabla \mathbf{H}\vert \vert \mathbf{v}_{s}\vert + \vert \nabla \mathbf{H}\vert \vert \mathbf{w}\vert+\vert \nabla \mathbf{v}\vert \vert \mathbf{H}_{s}\vert) \mathrm{d}\mathbf{x}\mathrm{d}s 
 \nonumber \\ 
 & \displaystyle \leq C \int _{0}^{t} \|\mathbf{H}\|_{L ^{\infty}(\Omega^{\epsilon})} \|\nabla \mathbf{v}\|_{L ^{4}(\Omega ^{\epsilon})} \|\nabla \mathbf{H}\|_{L ^{4}(\Omega ^{\epsilon})}(\|\mathbf{v}_{s}\|_{0,\Omega ^{\epsilon}}+\|\mathbf{w}\|_{0,\Omega ^{\epsilon}})\mathrm{d}s 
  \nonumber \\ 
  & \displaystyle \quad + C \int _{0}^{t} \|\mathbf{H}\|_{L ^{\infty}(\Omega^{\epsilon})} \|\nabla \mathbf{v}\|_{L ^{4}(\Omega ^{\epsilon})} \|\nabla \mathbf{v}\|_{L ^{4}(\Omega ^{\epsilon})} \|\mathbf{H}_{s}\|_{0,\Omega ^{\epsilon}}\mathrm{d}s 
   \nonumber \\ 
   & \displaystyle \leq \mathcal{M}_{0}+T ^{1/2}\mathcal{P}( \sup_{t \in [0,T]} \mathcal{E}^ \epsilon (t)).
\end{align}
For $ \mathcal{S}_{3} $, we integrate by parts to get 
\begin{align}
\displaystyle  |\mathcal{S} _3| & \leq C  \int_0^t \int_{ \Omega^\epsilon } |\mathbf{v}_{s} \cdot \mathbf{w}_s | \mathrm{d}x \mathrm{d}s + \left| \left. \int_{ \Omega^\epsilon } \mathbf{v}_{s} \cdot \mathbf{w}  \mathrm{d}\mathbf{x}\right|^t_0 \right|
\nonumber \\
& \displaystyle \leq \mathcal{M}_{0}+ C\int _{0}^{t}\int _{\Omega ^{\epsilon}}\vert \mathbf{v}_{s}\cdot \mathbf{w}\vert \mathrm{d}\mathbf{x}\mathrm{d}t + \delta \|\mathbf{v}_{t}(\cdot,t)\|_{0,\Omega ^{\epsilon}}^{2}+C _{\delta}\|\mathbf{w}(\cdot,t)\|_{0, \Omega ^{\epsilon}} ^{2} 
 \nonumber \\ 
 & \displaystyle \leq \mathcal{M} _{0}+C _{\delta} T ^{1/2}\mathcal{P}((\sup_{t \in [0,T]}  \mathcal{E}^ \epsilon (t)))+ \delta \|\mathbf{v}_{t}(\cdot,t)\|_{0, \Omega ^{\epsilon}}^{2}
\end{align}
for any $ \delta>0 $, where we have used \eqref{w-esti-small} and the Cauchy-Schwarz inequality. Similarly, we have for $ \mathcal{S}_{8} $ that 
\begin{align}\label{S-8-esti}
\displaystyle  \mathcal{S}_{8}& =\int_0^t \int _{ \partial\Omega ^{\epsilon}}\mathbf{H} ^{i}A _{i}^{\lambda}N _{\lambda}^{\epsilon}\partial _{s}H ^{j} w^{j}\mathrm{d}s -  \int_0^t \int _{ \Omega ^{\epsilon}}H ^{i}A _{i}^{\lambda},_{\lambda} \partial _{s}H ^{j} w^{j} \mathrm{d}\mathbf{x}\mathrm{d}s  
 \nonumber \\ 
 &\displaystyle \quad - \int_0^t \int _{ \Omega ^{\epsilon}}H ^{i},_{\lambda} A _{i}^{\lambda}\partial _{s}H ^{j} w^{j} \mathrm{d}\mathbf{x}\mathrm{d}s-\int_0^t \int _{ \Omega ^{\epsilon}}H ^{i}A _{i}^{\lambda}\partial _{s}H ^{j} w^{j},_{\lambda} \mathrm{d}\mathbf{x} \mathrm{d}s
  \nonumber \\ 
  &\displaystyle =-  \underbrace{ \int_0^t \int _{ \Omega ^{\epsilon}}H ^{i},_{\lambda}A _{i}^{\lambda}\partial _{s}H ^{j} w^{j} \mathrm{d}\mathbf{x}\mathrm{d}s}_{ \mathcal{S} _{8}^{a}}- \underbrace{ \int_0^t \int _{ \Omega ^{\epsilon}}H ^{i}A _{i}^{\lambda}\partial _{s}H ^{j} w^{j},_{\lambda} \mathrm{d}\mathbf{x}\mathrm{d}s}_{ \mathcal{S} _{8}^{b}},
\end{align}
where, thanks to the Sobolev embedding theorem, the H\"older inequality and the Young inequality, it holds that 
\begin{equation}\label{csj4}
\begin{aligned}
| \mathcal{S}_{8}^{a}| +  | \mathcal{S}_{8}^{b}|& \leq C  \int_0^t \int _{\Omega ^{\epsilon}} \vert \nabla \mathbf{H} \vert\vert \p_{s} \mathbf{H} \vert  \vert \mathbf{w} \vert \mathrm{d}\mathbf{x} \mathrm{d}s + C\int_0^t \int _{\Omega ^{\epsilon}} \vert  \mathbf{H}\vert\vert \p_{s} \mathbf{H} \vert  \vert \nabla \mathbf{w} \vert \mathrm{d}\mathbf{x}\mathrm{d}s \\
& \displaystyle \leq C \int_{0}^{t} \| \mathbf{H} \|_{2, \Omega ^{\epsilon}}  \| \mathbf{H}_s \|_{0,\Omega ^{\epsilon}}   \| \mathbf{w} \|_{1,\Omega ^{\epsilon}}\mathrm{d}s \\
& \displaystyle \leq \mathcal{M}_{0}+ C _{\delta}   T ^{1/2} \mathcal{P} (\sup_{t \in [0,T]}  \mathcal{E}^ \epsilon (t)) + \delta \sup_{t \in [0,T]} \mathcal{E}^ \epsilon (t)
\end{aligned}
\end{equation} 
for any $ \delta>0 $. Therefore we have from \eqref{S-8-esti} that 
\begin{align}
\displaystyle \mathcal{S} _{8} \leq \mathcal{M}_{0}+ C _{\delta}   T ^{1/2} \mathcal{P} (\sup_{t \in [0,T]}  \mathcal{E}^ \epsilon (t)) + \delta \sup_{t \in [0,T]} \mathcal{E}^ \epsilon (t).
\end{align}
The estimates for $ \mathcal{S}_{6} $ is straightforward. Indeed, it holds that 
\begin{align}
&\displaystyle  \mathcal{S}_{6}\leq C \int _{0}^{T}\int _{\Omega ^{\epsilon}}  \vert \mathbf{H}_{t}\vert \vert \nabla\mathbf{H}\vert\left( \vert \mathbf{v}_{t}\vert+ \vert \mathbf{w}\vert\right) \mathrm{d}\mathbf{x}\mathrm{d}t 
   \nonumber \\ 
   & \displaystyle \leq C \int _{0}^{T}  \|\mathbf{H}_{t}\|_{0,\Omega^\epsilon} \|\nabla \mathbf{H}\| _{L ^{4}(\Omega ^{\epsilon})}(\|\mathbf{v}_{t}\|_{L ^{4}(\Omega ^{\epsilon})}+\|\mathbf{w}\|_{L ^{4}(\Omega ^{\epsilon})}) \mathrm{d}t 
    \nonumber \\ 
    & \displaystyle \leq \mathcal{M}_{0}+ C _{\delta}   T  \mathcal{P} (\sup_{t \in [0,T]}  \mathcal{E}^ \epsilon (t)) + \delta \sup_{t \in [0,T]} \mathcal{E}^ \epsilon (t).
\end{align}
Finally, for $ \mathcal{S} _9$ and $ \mathcal{S} _{12}$, with a substitution of index, these two items can be combined to use the integration by parts in space again:
\begin{align*}
 \mathcal{S} _9 + \mathcal{S} _{12}&= \int_0^t \int _{ \partial\Omega ^{\epsilon}}H ^{i}A _{i}^{\lambda}N _{\lambda}^{\epsilon}\partial _{s}H ^{j} v _{s}^{j}\mathrm{d}\mathbf{x}\mathrm{d}s -  \int_0^t \int _{ \Omega ^{\epsilon}}\left( H ^{i}A _{i}^{\lambda},_{\lambda} \partial _{s}H ^{j} v _{s}^{j}+ H ^{i},_{\lambda}A _{i}^{\lambda}\partial _{s}H ^{j} v _{s}^{j}\right)  \mathrm{d}\mathbf{x}\mathrm{d}s \\&= \underbrace{  - \int_0^t  \int _{ \Omega ^{\epsilon}}H ^{i},_{\lambda} A _{i}^{\lambda}\partial _{s}H ^{j} v _{s}^{j} \mathrm{d}\mathbf{x}\mathrm{d}s}_{ \mathcal{S} _{13}},
\end{align*}
with  $\mathcal{S}_{13}$ being estimated, by \eqref{na-eta-bd}, the Sobolev embedding theorem, the H\"older inequality and the Young inequality, as 
\begin{align*}
| \mathcal{S}_{13}|   & \leq C  \int_0^t \int _{\Omega ^{\epsilon}} \vert \nabla \mathbf{H} \vert\vert \p_{s} \mathbf{H} \vert \vert \nabla \boldsymbol{\eta} \vert^2 \vert  \mathbf{v}_s \vert \mathrm{d}\mathbf{x}\mathrm{d}s
\\
&\displaystyle \leq C \int_{0}^{T} \| \mathbf{H} \|_{2, \Omega ^{\epsilon}}  \| \mathbf{H}_t \|_{0, \Omega ^{\epsilon}} \| \boldsymbol{\eta}\|_{3,\Omega ^{\epsilon}}^2  \| \mathbf{v}_t \|_{1, \Omega ^{\epsilon} } \mathrm{d}t\\
& \displaystyle \leq C _{\delta}      T \mathcal{P} (\sup_{t \in [0,T]}  \mathcal{E}^ \epsilon (t)) + \delta \sup_{t \in [0,T]} \mathcal{E}^ \epsilon (t) .
\end{align*}
Therefore we get for any $ \delta>0 $, 
\begin{align}
\displaystyle  \mathcal{S} _9 + \mathcal{S} _{12} \leq C _{\delta}      T \mathcal{P} (\sup_{t \in [0,T]}  \mathcal{E}^ \epsilon (t)) + \delta \sup_{t \in [0,T]} \mathcal{E}^ \epsilon (t).
\end{align}
Collecting all the estimates for $ \mathcal{S}_{i}\,(1,\cdots,12) $, we then get from \eqref{tim-sum-iden} that 
\begin{align}\label{tim-sumup-conclu}
&\displaystyle \| \mathbf{v}_t ( \cdot ,t) \|^2_{0, \Omega^\epsilon }+ \|\mathbf{H} _{t}\|_{0,\Omega ^{\epsilon}}^{2}+ \int_0^t\| \nabla \mathbf{v}_s\|^2_{0, \Omega^\epsilon } \mathrm{d} s 
 \nonumber \\ 
 &~\displaystyle \leq  \mathcal{M}_{0}+ C _{\delta}   T ^{1/2}  \mathcal{P} (\sup_{t \in [0,T]}  \mathcal{E}^ \epsilon (t)) + \delta \sup_{t \in [0,T]} \mathcal{E}^ \epsilon (t)+  C\sup_{t \in [0,T]} \mathcal{E}^ \epsilon (t)\int _{0}^{t}\|Q\|_{1, \Omega ^{\epsilon}}^{2}\mathrm{d}s 
\end{align}
for any $ \delta>0 $. To finish the proof of the present lemma, we control the last term on the right hand side of \eqref{tim-sumup-conclu}. We write equation (\ref{mhd-lagran}b) as
\begin{subequations}
\label{stokes_for_v}
\begin{alignat}{2}
-\Delta \mathbf{v} + \nabla Q & = \operatorname{div} [ (AA^T - \mathbb{I}_{3} ) \nabla \mathbf{v}]  - (A^T - \mathbb{I}_{3} ) \nabla Q- \mathbf{v}_t - \frac{1}{4 \pi} H ^{i}A _{i}^{\lambda} \mathbf{H}, _{\lambda}\ \ && \text{ in } \Omega^ \epsilon  \times (0,T] \,, \label{stokes_for_v.a} \\
\operatorname{div}  \mathbf{v} &= - (A^j_i - \delta ^j_i) v^i,_j \ \ && \text{ in } \Omega^\epsilon  \times [0,T] \,,  \\
\mathbf{v} & \in L^2(0,T; H^{2.5}(\Gamma^ \epsilon )\,. &&
\end{alignat}
\end{subequations}
According to Lemma \ref{lem-Stoke}, we have 
\begin{align}
\displaystyle \|\mathbf{v}\|_{2,\Omega ^{\epsilon}}+\| \nabla Q\|_{0,\Omega ^{\epsilon}} & \leq C \|\operatorname{div} [ (AA^T - \mathbb{I}_{3} ) \nabla \mathbf{v}]\|_{0, \Omega ^{\epsilon}}+C\|(A^T - \mathbb{I}_{3} ) \nabla Q\|_{0,\Omega ^{\epsilon}}+C\|\mathbf{v}_{t}\|_{0,\Omega ^{\epsilon}} 
  \nonumber \\ 
  & \displaystyle \quad +C\|H ^{i}A _{i}^{\lambda}\mathbf{H}, _{\lambda}\|_{0,\lambda}+C \|(A^j_i - \delta ^j_i) v^i,_j\| _{1,\Omega^\epsilon}+\vert \mathbf{v}\vert _{1.5, \Gamma ^{\epsilon}}.
\end{align}
This along with \eqref{na-eta-bd}, \eqref{A-AT-bd} and Lemma  \ref{lem-trace} further implies that 
\begin{align}\label{Q-1-esti}
 \displaystyle \| \nabla Q\|_{0,\Omega ^{\epsilon}} \leq C \left( \|\mathbf{v} _{t}\|_{0,\Omega^\epsilon}+ \|\mathbf{H}\|_{2,\Omega ^{\epsilon}}+\vert \mathbf{v}\vert _{1.5, \Gamma ^{\epsilon}} \right)  \leq C \sup _{t \in [0,T]}\mathcal{E}(t).   
 \end{align}
Next, take a function $ \chi $ such that 
\begin{align*}
\displaystyle - \Delta \chi =1\ \ \mbox{in } \Omega ^{\epsilon},\ \ \chi=0\ \ \mbox{on }\Gamma ^{\epsilon}.
\end{align*}
Then we get, thanks to \eqref{bd-te-l} and integration by parts,  
\begin{align}
\displaystyle  \int _{\Omega ^{\epsilon}}Q \mathrm{d}\mathbf{x} &=- \int _{\Omega ^{\epsilon}}\Delta \chi Q \mathrm{d}\mathbf{x} =- \int _{\Gamma ^{\epsilon}} \frac{\partial \chi}{\partial N ^{\epsilon}} Q \mathrm{d}S+ \int _{\Omega ^{\epsilon}}\nabla \chi \cdot \nabla Q \mathrm{d}\mathbf{x} 
 \nonumber \\ 
 & \displaystyle \leq C \left( \| Q\| _{0,\Gamma ^{\epsilon}}+ \|\nabla Q\|_{0,\Omega ^{\epsilon}} \right), 
\end{align}
where the constant $ C>0 $ is independent of $ \epsilon $. This along with \eqref{Q-1-esti} implies that 
\begin{align}\label{Q-L2}
\displaystyle \|Q\|_{0,\Omega ^{\epsilon}} \leq C \| Q\| _{0, \Gamma ^{\epsilon}}+ C \sup _{t \in [0,T]}\mathcal{E}(t).
\end{align}
Notice that 
\begin{gather*}
\displaystyle Q =  \mathbf{n}\cdot \left[ \text{Def}_{\boldsymbol{\eta}} \mathbf{v}\cdot \mathbf{n}\right] \ \ \  \text{ on } \Gamma^\epsilon  \times [0,T].
\end{gather*}
Then by the Sobolev embedding theorem, we get 
\begin{gather*}
\displaystyle  \| Q(\cdot,t)\| _{0,\Gamma ^{\epsilon}} \leq C \|\boldsymbol{\eta}(\cdot,t)\|_{3,\Omega ^{\epsilon}}^{3}\| \mathbf{v} (\cdot,t)\| _{2, \Omega ^{\epsilon}} \leq \mathcal{P}(\sup _{t \in [0,T]}\mathcal{E}^{\epsilon}(t))
\end{gather*}
for any $ t \in [0,T] $. Therefore we have from \eqref{Q-L2} that 
\begin{align}\label{Q-l2-con}
\displaystyle  \|Q\|_{0,\Omega ^{\epsilon}} \leq \mathcal{P}(\sup _{t \in [0,T]}\mathcal{E}^{\epsilon}(t)),
\end{align}
and thus, 
\begin{gather}\label{Q-h1-conclu}
\displaystyle \|Q\|_{1,\Omega ^{\epsilon}} \leq  \mathcal{P}(\sup _{t \in [0,T]}\mathcal{E}^{\epsilon}(t)).
\end{gather}
Combining \eqref{Q-h1-conclu} with \eqref{tim-sumup-conclu}, we then have for $ T<1 $ that 
 \begin{align}
 \displaystyle  \| \mathbf{v}_t ( \cdot ,t) \|^2_{0, \Omega^\epsilon }+ \|\mathbf{H} _{t}\|_{0,\Omega ^{\epsilon}}^{2}+ \int_0^t\| \nabla \mathbf{v}_s\|^2_{0, \Omega^\epsilon } \mathrm{d} s 
 \leq  \mathcal{M}_{0}+ C _{\delta}   T ^{1/2}  \mathcal{P} (\sup_{t \in [0,T]}  \mathcal{E}^ \epsilon (t)) + \delta \sup_{t \in [0,T]} \mathcal{E}^ \epsilon (t) .
 \end{align}
  This gives \eqref{tim-coclu-lem}.
\end{proof}

Now we are ready to verify the \emph{a priori} assumption \eqref{basic}, and then finish the proof of Proposition \ref{prop-key}. 
\begin{lemma}\label{lem-verify}
For sufficiently small $ T>0 $, it holds that 
\begin{align}
\displaystyle  \sup_{t \in [0,T]} \| \nabla \boldsymbol{\eta} (\cdot , t)- \mathbb{I}_{3} \|_{ L^\infty(\Omega^ \epsilon )} \le  \frac{1}{2}\vartheta^{10},\ \ \ \sup_{t \in [0,T]} \mathcal{E}^ \epsilon (t) \le 2 \mathcal{M}_0.
\end{align}
\end{lemma}
\begin{proof}
Applying Lemma \ref{lem-Stoke} to \eqref{tim-coclu-lem}, by using \eqref{Q-L2}, we have $\mathbf{v} \in L^\infty([0,T]; H^2(\Omega^ \epsilon ) ) \cap L^2(0,T; H^3( \Omega^\epsilon ))$ and
\begin{align*}
 & \sup_{t \in [0,T]}   \| \mathbf{v} ( \cdot ,t) \|^2_{2, \Omega^\epsilon }+ \int_0^T\| \mathbf{v}\|^2_{3, \Omega^\epsilon } \mathrm{d}t +  \int_0^T\| Q\|^2_{2, \Omega^\epsilon } \mathrm{d}t  \\
  &~ \leq \| \operatorname{div} [ (AA^T - \mathbb{I}_{3}) \nabla \mathbf{v}]  - (A^T - \mathbb{I}_{3} ) \nabla Q- \mathbf{v}_t - \frac{1}{\mu _{0}} H ^{i}A _{i}^{\lambda}\mathbf{H}, _{\lambda} \|^2_{1, \Omega^\epsilon } 
  \nonumber \\ 
  & ~ \quad\displaystyle +  \| (A^j_i - \delta ^j_i) v^i,_j  \|^2_{1, \Omega^\epsilon } + \int_0^T \vert \mathbf{v}  \vert^2_{2.5, \Gamma^\epsilon } \mathrm{d}t.
\end{align*}
Furthermore, the terms on the right side of the above estimation can be handled by using the Sobolev embedding theorem, the H\"older inequality and the Cauchy-Schwarz inequality. Therefore we get
\begin{align}\label{cs11}
&\sup_{t \in [0,T]}   \| \mathbf{v} ( \cdot ,t) \|^2_{2, \Omega^\epsilon }+ \int_0^T\| \mathbf{v}\|^2_{3, \Omega^\epsilon } \mathrm{d}t +  \int_0^T\| Q\|^2_{2, \Omega^\epsilon } \mathrm{d}t 
 \nonumber \\ 
 &~\displaystyle\leq 
\mathcal{M}_0 +     T \mathcal{P} (\sup_{t \in [0,T]}  \mathcal{E}^ \epsilon (t)) +
C\delta \sup_{t \in [0,T]} \mathcal{E}^ \epsilon (t) \,.
\end{align}
By this estimate together with \eqref{cs6} and \eqref{tim-coclu-lem}, we have
\begin{equation}
\sup_{t \in [0,T]} \mathcal{E}^ \epsilon (t)
 \le
M_0 +     T \mathcal{P} (\sup_{t \in [0,T]}  \mathcal{E}^ \epsilon (t)) + C\delta \sup_{t \in [0,T]} \mathcal{E}^ \epsilon (t)\,,
\end{equation}
By choosing $ \delta >0$ sufficiently small, we obtain that
\begin{equation}
\label{cs12}
\sup_{t \in [0,T]} \mathcal{E}^ \epsilon (t)
 \le
\mathcal{M}_0 +     T \mathcal{P} (\sup_{t \in [0,T]}  \mathcal{E}^ \epsilon (t))\,,
\end{equation}
where $ \mathcal{P} $ is a polynomial function which is independent of $ \epsilon $.
Then we need to prove $\mathcal{E}^ \epsilon (t)$ is a continuous function with respect to $t$. Recall that $v\in L^2(0,T; H^3(\Omega^ \epsilon ) )$ and $v_t\in L^2(0,T; H^1(\Omega^ \epsilon ) )$. By the definition of $\zeta_l$, we know $\zeta_l v\in L^2(0,T; H^3(\mathcal{B}_l ) )$ and $\zeta_l v_t\in L^2(0,T; H^1(\mathcal{B}_l ) )$ for $l=1,...,L$. Hence by summing over $l=1,...,L$, $ v\in C^0([0,T]; H^2( \Omega^\epsilon  ) )$.
With (\ref{mhd-lagran-c}), we have
\begin{align*}
\displaystyle   \int _{0}^{T}\|\partial _{t} \mathbf{H} \|^{2,}_{2,\Omega^\epsilon}\mathrm{d}t & = \int _{0}^{T}\|H ^{i}A _{i}^{k}\mathbf{v}, _{k}\| _{2,\Omega^\epsilon}^{2}\mathrm{d}t
 \nonumber \\
 & \displaystyle \leq C  \int _{0}^{T}\|\mathbf{H}\|_{2,\Omega^\epsilon}^{2}\|A\|_{2,\Omega^\epsilon}^{2}\|\mathbf{v}\|_{3,\Omega^\epsilon}^{2}\mathrm{d}t .
\end{align*}
This implies that $\partial _{t} \mathbf{H}  \in L^2(0,T; H^2(\Omega^\epsilon  ) )$ provided $ \mathbf{H} \in L ^{\infty}(0,T;H ^{2}(\Omega ^{\epsilon})) $, $ \boldsymbol{\eta} \in L ^{\infty}(0,T;H ^{3}(\Omega ^{\epsilon})) $ and $ \mathbf{v} \in L ^{2}(0,T;H ^{3}(\Omega ^{\epsilon})) $. Recall also that $ \mathbf{H} \in L ^{\infty}(0,T;H ^{2}(\Omega ^{\epsilon})) $. Then we have
\begin{align}
\displaystyle \mathbf{H} \in C([0,T];H ^{2}(\Omega ^{\epsilon})).
\end{align}
Since the pressure satisfies the elliptic system:
\begin{alignat*}{2}
-\Delta _{\boldsymbol{\eta}} Q & = v^i,_rA^r_j v^j,_sA^s_i+\frac{1}{4 \pi}H ^{i}, _{r}A _{j}^{r}H  ^{j}, _{s}A ^{s}_{i}\ \ && \text{ in } \Omega^ \epsilon  \times (0,T] \,,  \\
Q &=  \mathbf{n}\cdot \left[ \text{Def}_{\boldsymbol{\eta}} \mathbf{v}\cdot \mathbf{n}\right] \ \ && \text{ on } \Gamma^\epsilon  \times [0,T] \,,
\end{alignat*}
as in \cite{Shkoller-2019-APAN}, we then infer that $Q\in C([0,T]; H^1(\Omega^ \epsilon ) )$.  Then, using the momentum equation
 \eqref{stokes_for_v.a} and the magnetic field equation \eqref{mhd-lagran-c}, it follows that $\mathbf{v}_t\in C([0,T]; L^2(\Omega^ \epsilon ) )$ and $ \mathbf{H}_{t} \in C([0,T];L ^{2}(\Omega ^{\epsilon})) $. Thus we show that $\mathcal{E}^ \epsilon (t)$ is a continuous function with respect to $t$. Therefore we conclude that for sufficiently small $ T >0\, (\mbox{independent of }\epsilon)$, 
 \begin{align}\label{E-bd-conclu}
  \displaystyle  \sup_{t \in [0,T]} \mathcal{E}^ \epsilon (t) \le 2 \mathcal{M}_0.
  \end{align}
From \eqref{eta-eq}, we have  
  \begin{align*}
  \displaystyle  \|\nabla \boldsymbol{\eta}(\cdot,t)- \mathbb{I}_{3}\|_{2,\Omega ^{\epsilon}} \leq C\Big\|\int _{0}^{t}  \nabla \mathbf{v} \mathrm{d}s\Big\|_{2,\Omega ^{\epsilon}} \leq C \sqrt{t}\sup _{t \in [0,T]}\sqrt{\mathcal{E} ^{\epsilon}(t)}.
  \end{align*}
  This alongside Lemma \ref{lem-Sobolev} and \eqref{E-bd-conclu} implies that 
  \begin{gather}
  \displaystyle   \|\nabla \boldsymbol{\eta}(\cdot,t)- \mathbb{I}_{3}\|_{L ^{\infty}(\Omega ^{\epsilon})} \leq 2C \mathcal{M}_{0} \sqrt{t}.
  \end{gather}
  Therefore we get for for sufficiently small $ T >0\, (\mbox{independent of }\epsilon)$ that 
\begin{align*}
\displaystyle  \sup_{t \in [0,T]} \| \nabla \boldsymbol{\eta} (\cdot , t)- \mathbb{I}_{3} \|_{ L^\infty(\Omega^ \epsilon )} \le  \frac{1}{2}\vartheta^{10}.
\end{align*}
  The proof is complete.
\end{proof}

\section{Proof of the Main Theorem}\label{section8}
With the estimates established in the previous section, in this section we shall prove Theorem \ref{mainThm}. To do so,  a more quantitative estimate is needed  in order to claim the continuity of $\bar \p^2 v(t , \cdot )$ in  $ L^2( \Omega^\epsilon ) $. The proof is motivated by that for Proposition 7 in \cite{Shkoller-2019-APAN}. The new ingredient here is to deal with the coupling of the magnetic field with the velocity field and the evolution of the geometry and regularity of the free boundary. 

\begin{lemma}
For all $t\in [0,T]$,
\begin{equation}\label{est-main2}
\max_{s \in [0,t] } \| \bar \p^2 ( \mathbf{v} ( \cdot ,s )  - \mathbf{u}_0^ \epsilon )\|_{0, \Omega ^ \epsilon }^2
+ \int_0^t \| \bar \p^2 ( \mathbf{v}(\cdot ,s)  - \mathbf{u}_0^ \epsilon )\|_{1, \Omega ^ \epsilon }^2 \mathrm{d} s \lesssim t^{1/2} \mathcal{P} (\mathcal{M}_0).
\end{equation}
\end{lemma}

\begin{proof}
We write $\mathbf{v}(t) = \mathbf{v}( \cdot ,t)$ and set the constants in the system to be $ 1 $ for simplicity. The difference $\mathbf{v}(t) - \mathbf{u}_{0}^{\epsilon}$ satisfies the equation
\begin{equation}\nonumber
(\mathbf{v}- \mathbf{u}_{0}^{\epsilon})_t - \Delta _ {\boldsymbol{\eta}}(  \mathbf{v}- \mathbf{u}_{0}^{\epsilon}) + A^T \nabla Q = \Delta _{\boldsymbol{\eta}} \mathbf{u} _{0}^{\epsilon} + \mathbf{H} \cdot \nabla_{\boldsymbol{\eta}} \mathbf{H}.
\end{equation}
We localize the above equation to the boundary charts defined before and take the second-order tangential derivative as in lemma \ref{lem-bdy-regul}, then we get by integrating by parts that 
\begin{align}
&0=  {\frac{1}{2}} \frac{\mathrm{d}}{\mathrm{d}t} \| \zeta \bar \partial^2 [\mathbf{v}(t)- \mathbf{u}_{0}^{\epsilon}]\|^2_{0,B^+}
+  \int_{B^+}  \zeta ^2 \bar \p^2 [ A^k_{\lambda} A^j_{\lambda} (\mathbf{v}- \mathbf{u}_{0}^{\epsilon}),_j] \cdot  \bar\p^2[ ( \mathbf{v}- \mathbf{u} _{0}^{\epsilon})],_k \mathrm{d} \mathbf{y}  \nonumber \\
& ~\quad+  \int_{B^+}  \zeta ^2\bar \p^2 [A^k_i  Q]\, \bar\p^2  v^i,_k \mathrm{d}\mathbf{y}  +  \int_{B^+}  \zeta ^2\bar \p^2 [ A^k_{\lambda} A^j_{\lambda} {\mathbf{u}_{0}^{\epsilon}},_j] \cdot  \bar\p^2[  (\mathbf{u}- \mathbf{u}_{0}^{\epsilon})],_k \mathrm{d}\mathbf{y} 
 \nonumber \\ 
 & ~\displaystyle \quad + \int_{B^+} \zeta^2  \bar\p^2 (H^i A_{i}^{\lambda}  H,_{\lambda}^{j}) \bar\p^2 v^{j} \mathrm{d} \mathbf{y}.
 \end{align}
 Integrating the above equation over the time interval $[0,T]$, we get 
\begin{align}\label{conti-iden}
  & \| \zeta \bar \partial^2 [\mathbf{v}(t)- \mathbf{u}_{0}^{\epsilon}]\|^2_{0,B^+} + \int_0^t   \| \zeta \bar\p^2   [\mathbf{v}(s)- \mathbf{u}_{0}^{\epsilon} ]  \|^2_{1,B^+}\mathrm{d}s \leq \underbrace{\Big| \int_0^t \int_{B^+} \bar \p^2 [A^k_i  Q]\, \bar \partial^2[ \zeta ^2 ( v- u_{0}^{\epsilon})^i],_k \mathrm{d} \mathbf{y} \mathrm{d} s\Big |}_{ \mathcal{T} _{1}} 
   \nonumber \\ 
   &~\displaystyle +  \underbrace{\Big|\int_0^t  \int_{B^+}   \bar \partial^2 [ (A^k_{\lambda} A^j_{\lambda} - \delta^{kj}) (\mathbf{v}- \mathbf{u}_{0}^{\epsilon}),_j] \cdot  \bar \partial^2[ \zeta ^2 ( \mathbf{v}- \mathbf{u}_{0}^{\epsilon})],_k  \mathrm{d} \mathbf{y} \mathrm{d} s\Big|}_{ \mathcal{T} _{2}}  \\
  &~ +  \underbrace{\Big|\int_0^t  \int_{B^+}   \bar \p^2  ( v- u _{0}^{\epsilon})^i,_k\, [\ [ \bar \partial^2 \zeta ^2 ( v- u _{0}^{\epsilon})^i  + 2 \bar \p \zeta^2 \bar \p (v-\uu)^i],_k\textcolor{black} {+\zeta^2,_k \bar\p^2 v^i]}  \mathrm{d}\mathbf{y} \mathrm{d} s \Big|}_{ \mathcal{T} _{3}}   
   \nonumber \\ 
   & \displaystyle ~ +   \underbrace{\Big|\int_0^t  \int_{B^+}  \bar \partial^2 [ (A^k_{\lambda} A^j_{\lambda}  {\mathbf{u}_{0}^{\epsilon}},_j] \cdot  \bar \partial^2[ \zeta ^2 (\mathbf{v}- \mathbf{u}_{0}^{\epsilon})],_k  \mathrm{d} \mathbf{y} \mathrm{d} s \Big|}_{ \mathcal{T} _{4}}+ \underbrace{\Big|\int_0^t  \int_{B^+} \zeta^2  \bar \partial^2 (H^i A_{i}^{\lambda}  H^{j},_{\lambda}) \bar \partial^2 v^{j}   \mathrm{d} \mathbf{y} \mathrm{d} s  \Big|}_{ \mathcal{T} _{5}},
\end{align}
The estimates of $\mathcal{T} _{1} \thicksim \mathcal{T} _{4}$ can be derived by similar arguments as in the proof of Proposition 7.1 of \cite{Shkoller-2019-APAN}. So we quote the conclusion in the following without proof.
\begin{align}\label{T-1-4-ESTI}
\displaystyle \sum _{i=1}^{4}\mathcal{T}_{i} \leq \sqrt{t}\mathcal{P}(\mathcal{M}_{0}). 
\end{align}
Next we shall estimate $ \mathcal{T}_{5} $. It follows from integration by parts that 
\begin{align}
\displaystyle \mathcal{T}_{5} \leq  \Big\vert \int _{0}^{t}\int _{B ^{+}}\zeta ^{2}\bar{\partial}^{2}(H ^{i}A _{i}^{\lambda})H ^{j}, _{\lambda}\bar{\partial}^{2}v ^{j}\mathrm{d}\mathbf{y}\mathrm{d}s \Big\vert+ \Big\vert \int _{0}^{t}\int _{B ^{+}}(\zeta ^{2}H ^{i}\bar{\partial}^{2}v ^{j}), _{\lambda}A _{i}^{\lambda}\bar{\partial}^{2}H ^{j}\mathrm{d}\mathbf{y}\mathrm{d}s \Big\vert,
\end{align}
which along with \eqref{bd-te-l}, \eqref{na-eta-bd}, \eqref{E-bd-conclu}, the H\"older inequality and the Sobolev embedding theorem implies that 
\begin{align}\label{T-5-esti}
\displaystyle  \mathcal{T}_{5} & \leq C\int _{0}^{t}\int _{\Omega ^{\epsilon}}\left( \sum _{k=1}^{2}(\vert \nabla^{k} \mathbf{H}\vert+ \vert \nabla \mathbf{H}\vert \vert \nabla ^{k} \boldsymbol{\eta}\vert)   +\sum _{k=2}^{3}\vert \mathbf{H}\vert \vert \nabla ^{k}\boldsymbol{\eta}\vert++\vert \mathbf{H}\vert \vert \nabla ^{2}\boldsymbol{\eta}\vert ^{2}\right)  (\sum _{k=1}^{2}\vert \nabla \mathbf{H}\vert \vert \nabla ^{k}\mathbf{v}\vert) \mathrm{d}\mathbf{x}\mathrm{d}s 
 \nonumber \\ 
 & \displaystyle \quad+ C \int _{0}^{t}\int _{\Omega ^{\epsilon}}\left(  \sum _{k=1}^{3}\vert \mathbf{H}\vert \vert \nabla ^{k}\mathbf{v}\vert+\sum _{k=1}^{2}\vert \nabla \mathbf{H} \vert\vert \nabla ^{k}\mathbf{v}\vert\right) (\sum _{k=1}^{2}\vert \nabla ^{k}\mathbf{H}\vert )\mathrm{d}\mathbf{x}\mathrm{d}s 
  \nonumber \\ 
  & \displaystyle \leq C \int _{0}^{t} \left( \|\mathbf{H}\|_{2, \Omega ^{\epsilon}} ^{2}\|\mathbf{v}\|_{2, \Omega ^{\epsilon}}+ \| \mathbf{H}\|_{2, \Omega ^{\epsilon}} ^{2} \|\boldsymbol{\eta} \|_{3, \Omega ^{\epsilon}} \|\mathbf{v} \|_{3, \Omega ^{\epsilon}} +  \|\mathbf{H}\| _{2, \Omega ^{\epsilon}}^{2}\|\mathbf{v}\|_{3, \Omega ^{\epsilon}} \right) \mathrm{d}s 
   \nonumber \\ 
   & \displaystyle \leq C t ^{1/2}\mathcal{P}(\mathcal{M}_{0}).
\end{align}
Collecting the estimates in \eqref{T-1-4-ESTI} and \eqref{T-5-esti}, we get from \eqref{conti-iden} that 
\begin{align}
\displaystyle \| \zeta \bar \partial^2 [\mathbf{v}(t)- \mathbf{u}_{0}^{\epsilon}]\|^2_{0,B^+} + \int_0^t   \| \zeta \bar\p^2   [\mathbf{v}(s)- \mathbf{u}_{0}^{\epsilon} ]  \|^2_{1,B^+}\mathrm{d}s  \leq C t ^{1/2}\mathcal{P}(\mathcal{M}_{0}) 
\end{align}
for any $ t \in [0,T] $, where $ C>0 $ is a constant independent of $ \epsilon $. The proof is complete.
\end{proof}

With (\ref{est-main2}) and the estimates obtained in Sec. \ref{sec-a_priori-esti}, we are now ready to prove the main theorem of the paper. We remark that analysis is similar to the one in Section $ 8 $ of \cite{Shkoller-2019-APAN}, and for the sake of completeness, we will briefly carry out the proof in the following. To begin with, since $\operatorname{div} \mathbf{u}_0^\epsilon =0$ and $ \mathop{\mathrm{div}}\nolimits \mathbf{v}=- (A _{i}^{j}- \delta ^{ij})v ^{i}, _{j} $, it holds that 
\begin{align}\label{est-v-u0}
 & \| \bar \p \operatorname{div}  (\mathbf{v} - \mathbf{u}_0^\epsilon) \|^2_{0, \Omega^\epsilon }=\| \bar \p \operatorname{div}  \mathbf{v}  \|^2_{0, \Omega^\epsilon }=\| \bar \p ((A _{i}^{j}- \delta ^{ij})v ^{i}, _{j})  \|^2_{0, \Omega^\epsilon }\\
  &~ \leq \| \bar \p (A- \mathbb{I}_3 ) \, \nabla \mathbf{v} \|^2_{0, \Omega^\epsilon }+  \| (A- \mathbb{I}_3 ) \, \bar \p \nabla \mathbf{v} \|^2_{0, \Omega^\epsilon } \\
 & ~ \leq \sqrt{T} \mathcal{P} (\mathcal{M}_0) .
\end{align}
By \eqref{est-main2}, \eqref{est-v-u0} and the normal trace theorem, we see that $\bar \p^2 (\mathbf{v} - \mathbf{u}_{0}^{\epsilon}) \cdot \mathbf{N}^ \epsilon   \in C([0,T; H^{ - {\frac{1}{2}} }(\Gamma^ \epsilon ))$ and $ \| \bar \p^2 (\mathbf{v} - \mathbf{u}_{0}^{\epsilon}) \cdot \mathbf{N}^\epsilon\|^2_{ -1/2, \Gamma^ \epsilon } \le \sqrt{T} \mathcal{P} (\mathcal{M}_0)$, then we have
\begin{equation}\nonumber
 \| (\mathbf{v} - \mathbf{u}_{0}^{\epsilon}) \cdot \mathbf{N}^\epsilon\|^2_{ 1.5, \Gamma^ \epsilon } \le \sqrt{T} \mathcal{P} (\mathcal{M}_0) \,,
\end{equation}
and hence by Lemma \ref{lem-Sobolev}, we have 
\begin{equation}\label{cs200}
\max_{x \in \Gamma^ \epsilon } \vert (\mathbf{v} - \mathbf{u}_{0}^{\epsilon}) \cdot \mathbf{N}^\epsilon\vert \le T ^{\frac{1}{4}}  \mathcal{P} (\mathcal{M}_0) 
\end{equation}
for any $ t \in [0,T] $. According to the construction of $\mathbf{u}_0^\epsilon$ in Sec. \ref{sec-initial_data}, we have $\mathbf{u}_0^\epsilon ( \mathbf{X}_+^ \epsilon ) \cdot \mathbf{N}^\epsilon = -1$ and $ \mathbf{u}_0^\epsilon ( \mathbf{X}_-)\cdot \mathbf{N}^ \epsilon =0 $. Let $T>0$ be as in Lemma \ref{lem-verify}. And let $9\epsilon < T$ and $\mathbf{Y}$ be a point on $\partial \omega_- \cap \{x_3=0\}$. Then we will do some analysis about the locations of $\mathbf{X}_+^\epsilon$ and $\mathbf{Y}$ at $t=9\epsilon$.
For $ \mathbf{X}_+^ \epsilon$, by (\ref{cs200}) we have
\begin{align*}\label{X11}
\eta^3(\mathbf{X}_+^ \epsilon , 9 \epsilon ) =\boldsymbol{\eta}(\mathbf{X}_+^ \epsilon ,9\epsilon)  \cdot \mathbf{e}_3 =\epsilon + \int_0^{9\epsilon} v^3(\mathbf{X}_+^ \epsilon , s)\mathrm{d} s < -7\epsilon,
\end{align*}
\begin{equation}\label{X12}
|\eta^ \alpha (\mathbf{X}_+^ \epsilon , t )| \le10 \epsilon \mathcal{P}(\mathcal{M}_0) \ \ \alpha =1,2, \\ \forall t \in[0,9\epsilon],
\end{equation}
where $ \mathbf{e}_{3}=(0,0,1) \in \mathbb{R}^{3} $. For $\mathbf{Y}$, we have

\begin{align*}\label{Y11}
\eta^3( \mathbf{Y}, 9 \epsilon ) = \int_0^{9\epsilon} v^3(\mathbf{Y},s)\mathrm{d} s \geq 9^{\frac{5}{4}}\mathcal{P}(\mathcal{M}_0) \epsilon^{\frac{5}{4}},
\end{align*}
\begin{equation}\label{Y12}
| \eta^ \alpha (\mathbf{Y} , t ) -Y^ \alpha | \le 10 \epsilon \mathcal{P}(\mathcal{M}_0) \ \ \alpha =1,2,  \\ \forall t \in[0,9\epsilon].
\end{equation}
From the above analysis, we know that at $t=0$, $\mathbf{X}_+^\epsilon$ is exactly above $\mathbf{Y}$. However, at $t=9\epsilon$, $\boldsymbol{\eta}(\mathbf{X}_+^\epsilon,9\epsilon)$ is vertically below $\boldsymbol{\eta} (\partial \omega_{-} \cup \{ x_3=0 \} ,9\epsilon)$. So there must exist a time $0<T_1<9\epsilon$ and $\mathbf{Y}\in \partial \omega_{-} \cup \{ x_3=0 \}$ such that $\boldsymbol{\eta}(\mathbf{X}_+^\epsilon,T_1)=\boldsymbol{\eta}(\mathbf{Y},T_1 )$. Then the proof of the main theorem is completed.


\section*{Acknowledgments}
 The work of Guangyi Hong was partially supported by the National Natural Science Foundation $\#$ 12201221, and the Guangdong Basic and Applied Basic Research Foundation $\#$ 2021A1515111038. Luo's research is supported by a grant from the Research Grants Council of the Hong Kong Special Administrative Region, China (Project No. 11307420).


\begin{thebibliography}{10}

\bibitem{Alazard-2014}
{\sc T.~Alazard, N.~Burq, and C.~Zuily}, {\em On the {C}auchy problem for gravity water waves}, Invent. Math., 198 (2014), pp.~71--163.

\bibitem{Amrouche-stokes-2011}
{\sc C.~Amrouche and N.~E.~H. Seloula}, {\em On the {S}tokes equations with the {N}avier-type boundary conditions}, Differ. Equ. Appl., 3 (2011),
  pp.~581--607.

\bibitem{Beale-1981}
{\sc J.~T. Beale}, {\em The initial value problem for the {N}avier-{S}tokes equations with a free surface}, Comm. Pure Appl. Math., 34 (1981),
  pp.~359--392.

\bibitem{Fefferman-2013-Ann-math}
{\sc A.~Castro, D.~C\'{o}rdoba, C.~Fefferman, F.~Gancedo, and J.~G\'{o}mez-Serrano}, {\em Finite time singularities for the free boundary incompressible {E}uler equations}, Ann. of Math. (2), 178 (2013), pp.~1061--1134.

\bibitem{Castro-2019-AnnPDE}
{\sc A.~Castro, D.~C\'{o}rdoba, C.~Fefferman, F.~Gancedo, and
  J.~G\'{o}mez-Serrano}, {\em Splash singularities for the free boundary
  {N}avier-{S}tokes equations}, Ann. PDE, 5 (2019), pp.~Paper No. 12, 117.
  
\bibitem{CWY} 
{\sc G.-Q.~Chen and Y.-G.~Wang}, {\em Existence and stability of compressible current-vortex sheets in three-dimensional magnetohydrodynamics}, Arch. Ration. Mech. Anal. 187 (2008), pp. 369--408.

\bibitem{Shkoller-2010-Sima}
{\sc C.~H.~A. Cheng and S.~Shkoller}, {\em The interaction of the 3{D}
  {N}avier-{S}tokes equations with a moving nonlinear {K}oiter elastic shell},
  SIAM J. Math. Anal., 42 (2010), pp.~1094--1155.

\bibitem{Cheng-Shkoller-2017-JMFM}
{\sc C.~H.~A. Cheng and S.~Shkoller}, {\em Solvability and regularity for an
  elliptic system prescribing the curl, divergence, and partial trace of a
  vector field on {S}obolev-class domains}, J. Math. Fluid Mech., 19 (2017),
  pp.~375--422.
  
  
 
\bibitem{Chri-Lindblad-2000}
{\sc D.~Christodoulou and H.~Lindblad}, {\em On the motion of the free surface
  of a liquid}, Comm. Pure Appl. Math., 53 (2000), pp.~1536--1602.
  
\bibitem{CMST} 
{\sc J.F.~Coulombel, A.~Morando, P.~Secchi, and P.~Trebeschi}, {\em A priori estimates for 3D incompressible current-vortex sheets}, Commun. Math. Phys. 311 (2012), pp. 247--275. 



\bibitem{Shkoller-2007-Jams}
{\sc D.~Coutand and S.~Shkoller}, {\em Well-posedness of the free-surface
  incompressible {E}uler equations with or without surface tension}, J. Amer.
  Math. Soc., 20 (2007), pp.~829--930.

\bibitem{Shkoller-2014-splash}
{\sc D.~Coutand and S.~Shkoller}, {\em On the finite-time splash and splat
  singularities for the 3-{D} free-surface {E}uler equations}, Comm. Math.
  Phys., 325 (2014), pp.~143--183.

\bibitem{Coutand-Shkoller-2016-ARMA}
{\sc D.~Coutand and S.~Shkoller}, {\em On the impossibility of finite-time
  splash singularities for vortex sheets}, Arch. Ration. Mech. Anal., 221
  (2016), pp.~987--1033.

\bibitem{Shkoller-2019-APAN}
{\sc D.~Coutand and S.~Shkoller}, {\em On the splash singularity for the
  free-surface of a {N}avier-{S}tokes fluid}, Ann. Inst. H. Poincar\'{e} C
  Anal. Non Lin\'{e}aire, 36 (2019), pp.~475--503.
  
\bibitem{DJJ}
{\sc R.~Duan, F.~Jiang, and S.~Jiang}, {\em On the Rayleigh-Taylor instability for incompressible, inviscid magnetohydrodynamic flows}, SIAM J. Appl. Math., 71 (2011), pp. 1990--2013.

  
  

\bibitem{marcati-oldroyd-B-NoDEA-2017}
{\sc E.~Di~Iorio, P.~Marcati, and S.~Spirito}, {\em Splash singularities for a
  2{D} {O}ldroyd-{B} model with nonlinear {P}iola-{K}irchhoff stress}, NoDEA
  Nonlinear Differential Equations Appl., 24 (2017), pp.~Paper No. 60, 20.

\bibitem{marcati-oldroyd-B-ARMA}
{\sc E.~Di~Iorio, P.~Marcati, and S.~Spirito}, {\em Splash singularities for a
  general {O}ldroyd model with finite {W}eissenberg number}, Arch. Ration.
  Mech. Anal., 235 (2020), pp.~1589--1660.

\bibitem{marcati-Advance-2020}
{\sc E.~Di~Iorio, P.~Marcati, and S.~Spirito}, {\em Splash singularity for a
  free-boundary incompressible viscoelastic fluid model}, Adv. Math., 368
  (2020), pp.~107124, 64.

\bibitem{Fefferman-2016-Duke-twophase}
{\sc C.~Fefferman, A.~D. Ionescu, and V.~Lie}, {\em On the absence of splash
  singularities in the case of two-fluid interfaces}, Duke Math. J., 165
  (2016), pp.~417--462.

\bibitem{gu2021local}
{\sc X.~Gu, C.~Luo, and J.~Zhang}, {\em Local well-posedness of the
  free-boundary incompressible magnetohydrodynamics with surface tension},
  arXiv preprint arXiv:2105.00596, (2021).

\bibitem{Gu-Wangyanjin-JMPA}
{\sc X.~Gu and Y.~Wang}, {\em On the construction of solutions to the
  free-surface incompressible ideal magnetohydrodynamic equations}, J. Math.
  Pures Appl. (9), 128 (2019), pp.~1--41.

\bibitem{Guiguilong-Peking-2021}
{\sc G.~Gui}, {\em Lagrangian approach to global well-posedness of the viscous
  surface wave equations without surface tension}, Peking Math. J., 4 (2021),
  pp.~1--82.

\bibitem{Guo-Tice-2013-local}
{\sc Y.~Guo and I.~Tice}, {\em Local well-posedness of the viscous surface wave
  problem without surface tension}, Anal. PDE, 6 (2013), pp.~287--369.

\bibitem{Hao-Luo-2014-ARMA}
{\sc C.~Hao and T.~Luo}, {\em A priori estimates for free boundary problem of
  incompressible inviscid magnetohydrodynamic flows}, Arch. Ration. Mech.
  Anal., 212 (2014), pp.~805--847.

\bibitem{Hao-Luo-2020-CMP}
{\sc C.~Hao and T.~Luo}, {\em Ill-posedness of free boundary problem of the
  incompressible ideal {MHD}}, Comm. Math. Phys., 376 (2020), pp.~259--286.

\bibitem{Hao-TAO-2021-JDE}
{\sc C.~Hao and T.~Luo}, {\em Well-posedness for the linearized free boundary
  problem of incompressible ideal magnetohydrodynamics equations}, J.
  Differential Equations, 299 (2021), pp.~542--601.

\bibitem{hao2023splash}
{\sc C.~Hao and S.~Yang}, {\em Splash singularity for the free boundary
  incompressible viscous mhd}, arXiv preprint arXiv:2304.06893, (2023).
  
  \bibitem{JJW}{\sc F.~Jiang, S.~Jiang, and Y.~Wang}, {\em On the Rayleigh-Taylor instability for the incompressible viscous magnetohydrodynamic equations}, Comm. Partial Differential Equations, 39 (2014), pp.~399--438. 



\bibitem{David-Lannes-2005}
{\sc D.~Lannes}, {\em Well-posedness of the water-waves equations}, J. Amer.
  Math. Soc., 18 (2005), pp.~605--654.

\bibitem{Lindblad-2005-Ann-of-Math}
{\sc H.~Lindblad}, {\em Well-posedness for the motion of an incompressible
  liquid with free surface boundary}, Ann. of Math. (2), 162 (2005),
  pp.~109--194.
  
  \bibitem{ST} {\sc P. Secchi and Y. Trakhinin}, {\em Well-posedness of the plasma-vacuum interface problem}, Nonlinearity, 27 (2014), pp. 105--169.


\bibitem{Zeng-chognchun-Shatah-2008-CPAM}
{\sc J.~Shatah and C.~Zeng}, {\em Geometry and a priori estimates for free
  boundary problems of the {E}uler equation}, Comm. Pure Appl. Math., 61
  (2008), pp.~698--744.

\bibitem{Solonnikov-1977}
{\sc V.~A. Solonnikov}, {\em Solvability of the problem of the motion of a
  viscous incompressible fluid that is bounded by a free surface}, Izv. Akad.
  Nauk SSSR Ser. Mat., (1977), pp.~1388--1424, 1448.

\bibitem{Solonnikov-1992}
{\sc V.~A. Solonnikov}, {\em Solvability of a problem on the evolution of a
  viscous incompressible fluid bounded by a free surface on a finite time
  interval}, St. Petersburg Math. J., 3 (1992), pp.~189--220.
  
\bibitem{SWZ} 
{\sc Y.~Sun, W.~Wang, and Z.~Zhang}, {\em Nonlinear stability of current-vortex sheet to the incompressible MHD equations}, Commun. Pure Appl. Math., 71 (2018), pp. 356--403.
  
\bibitem{SWZ1}
{\sc Y.~Sun, W.~Wang, and Z.~Zhang}, {\em Well-posedness of the plasma-vacuum interface problem for ideal incompressible MHD}, Arch. Ration. Mech. Anal., 234 (2019), pp. 81--113.

\bibitem{T1}
{\sc Y.~Trakhinin}, {\em Existence of compressible current-vortex sheets: variable coefficients linear analysis}, Arch.
Ration. Mech. Anal., 177 (2005), pp. 331--366.

\bibitem{T2}
{\sc Y.~Trakhinin}, {\em The existence of current-vortex sheets in ideal compressible magnetohydrodynamics}, Arch.
Ration. Mech. Anal., 191 (2009), pp. 245--310.

\bibitem{T3}
{\sc Y.~Trakhinin and T.~Wang}, {\em Nonlinear stability of MHD contact discontinuities with surface tension}, Arch.
Ration. Mech. Anal., 243 (2022), pp. 1091--1149.
  
\bibitem{TW1} 
{\sc Y.~Trakhinin and T.~Wang}, {\em Well-posedness of free boundary problem in non-relativistic and relativistic ideal compressible magnetohydrodynamics}, Arch. Ration. Mech. Anal., 239 (2021), pp. 1131--1176. 
 
\bibitem{TW2} 
{\sc Y.~Trakhinin and T.~Wang}, {\em Well-posedness for the free-boundary ideal compressible magnetohydrodynamic equations with surface tension}, Math. Ann., 383 (2022), pp. 761--808.

\bibitem{WX1} 
{\sc Y.~Wang and Z.~Xin}, {\em Global well-posedness of free interface problems for the incompressible inviscid resistive MHD}, Comm. Math. Phys., 388 (2021), pp. 1323--1401.

\bibitem{WX2} 
{\sc Y.~Wang and Z.~Xin}, {\em Existence of Multi-dimensional Contact Discontinuities for the Ideal Compressible Magnetohydrodynamics}, arXiv:2112.08580, (2022). 


\bibitem{wusijue-1997-Invention}
{\sc S.~Wu}, {\em Well-posedness in {S}obolev spaces of the full water wave
  problem in {$2$}-{D}}, Invent. Math., 130 (1997), pp.~39--72.

\bibitem{wusijue-1999-JAMS}
{\sc S.~Wu}, {\em Well-posedness in {S}obolev spaces of the full water wave
  problem in 3-{D}}, J. Amer. Math. Soc., 12 (1999), pp.~445--495.

\bibitem{zhangping-zhangzhifei-2008-CPAM}
{\sc P.~Zhang and Z.~Zhang}, {\em On the free boundary problem of three-dimensional incompressible {E}uler equations}, Comm. Pure Appl. Math.,
  61 (2008), pp.~877--940.

\end{thebibliography}

\end{document}